\documentclass[12pt]{amsart}
\usepackage{latexsym,amssymb,enumerate,amsmath}
\newtheorem*{theoremA}{Theorem A}
\newtheorem*{theoremA'}{Theorem A'}
\newtheorem*{theoremA"}{Theorem A"}
\newtheorem*{theoremB}{Theorem B}
\newtheorem*{conjecture}{Conjecture}

\newtheorem{theorem}{Theorem}[section]
\newtheorem{corollary}[theorem]{Corollary}
\newtheorem{lemma}[theorem]{Lemma}
\newtheorem{proposition}[theorem]{Proposition}

\newtheorem{remark}[theorem]{Remark}

\def\gen#1{\langle#1\rangle}

\def\pd#1#2{{\rm ppd}(#1, #2)}
\def\bpd#1#2{{\rm bppd}(#1, #2)}
\def\lbpd#1#2{{\rm Lbpd}(#1, #2)}
\def\lpd#1#2{{\rm Lpd}(#1, #2)}

\def\GL{{\mathrm{GL}}}
\def\SU{{\mathrm{SU}}}
\def\PSL{{\mathrm{PSL}}}
\def\PSU{{\mathrm{PSU}}}
\def\PSp{{\mathrm{PSp}}}
\def\POm{{\mathrm{P}\Omega}}
\def\SL{{\mathrm{SL}}}
\def\Sp{{\mathrm{Sp}}}
\def\ppd{{\mathrm{ppd}}}
\def\bbppd{{\mathrm{bppd}}}
\def\Lbpd{{\mathrm{Lbpd}}}

\begin{document}

\title{A new solvability criterion for finite groups}

\author[S. Dolfi]{Silvio Dolfi}
\address{Silvio Dolfi, Dipartimento di Matematica U. Dini,\newline
Universit\`a degli Studi di Firenze, viale Morgagni 67/a,
50134 Firenze, Italy.}
\email{dolfi@math.unifi.it}

\author[M. Herzog]{Marcel Herzog}
\address{Marcel Herzog, Department of Mathematics,\newline
Raymond and Beverly Sackler Faculty of Exact Sciences,\newline 
Tel Aviv University,  Tel Aviv, Israel.}
\email{herzogm@post.tau.ac.il}

\author[C. Praeger]{Cheryl E. Praeger}
\address{Cheryl E. Praeger, School of Mathematics and Statistics,\newline
The University of Western Australia,
35 Stirling Highway, Crawley, WA 6009, Australia}
\email{praeger@maths.uwa.edu.au}

\thanks{The first author is grateful to the School of Mathematics and Statistics of the
University of Western Australia for its hospitality and support, while the investigation
was carried out. He was partially supported by the
MIUR project ``Teoria dei gruppi e applicazioni''. The third author was supported by 
Federation Fellowship FF0776186 of the Australian Research Council.}

\begin{abstract}
In 1968, John Thompson proved that a finite group $G$ is solvable if and only if every 
$2$-generator subgroup of $G$ is solvable. In this paper, we prove that solvability of a finite
group $G$ is guaranteed by a seemingly weaker condition: $G$ is solvable if for all conjugacy classes $C$ 
and $D$ of $G$, \emph{there exist} $x\in C$ and $y\in D$ for which $\gen{x,y}$ is solvable. We also prove 
the following property of finite nonabelian simple groups, which is the key tool for our proof of the 
solvability criterion: if $G$ is a finite nonabelian simple group, then there exist two integers $a$ and $b$
which represent orders of elements in $G$ and for all elements $x,y\in G$ with $|x|=a$ and $|y|=b$, the subgroup
$\gen{x,y}$ is nonsolvable. 
\end{abstract}
\subjclass[2000]{20D10, 20F16}
\keywords{Solvable groups, finite simple groups}
\maketitle

\section{Introduction}

John G. Thompson's famous `N-group paper' ~\cite{T} of 1968  included the following important solvability criterion for finite groups:
\begin{quote}
\emph{A finite group is solvable if and only if every pair of its elements generates a solvable group.}
\end{quote}
P. Flavell~\cite{F} gave a relatively simple  proof of Thompson's result   in 1995.
We prove that solvability of finite groups is guaranteed by 
a seemingly weaker condition than the solvability of all its $2$-generator
subgroups. 

\begin{theoremA}
Let $G$ be a finite group such that, for all $x,y \in G$, there
exists an element $g \in G$ for which $\gen{x,y^g}$ is solvable.
Then $G$ is solvable.
\end{theoremA}
Theorem~A  can be rephrased as the following equivalent result.

\begin{theoremA'}
Let $G$ be a finite group such that, for all conjugacy classes
$C$ and $D$ of $G$ (possibly $C=D$), there exist $x\in C$ and $y\in D$ for which
$\gen{x, y}$ is solvable. Then $G$ is solvable.
\end{theoremA'}

Our second main result, which is the key tool for proving Theorem~A,
deals with the nonsolvability of certain $2$-generator subgroups 
of finite nonabelian simple groups. For a finite group $G$ let
\begin{equation}\label{oe}
oe(G)= \{ m\ |\ \exists\, g \in G \hbox{ with } |g| = m\} 
\notag\end{equation}
denote the set  of element orders of $G$. Using the classification
of finite simple groups, we prove the following theorem.

\begin{theoremB}
Let $G$ be a finite nonabelian simple group. Then there exist
$a,b \in oe(G)$,  such that, for all $x,y\in G$ with
$|x|=a$, $|y|= b$, the subgroup $\gen{x,y}$ is nonsolvable.
\end{theoremB}
 
\par
Theorem~B was proved separately for alternating groups,
sporadic groups, classical groups of Lie type and exceptional groups
of Lie type in Propositions~\ref{prop3.4}, \ref{prop3.5}, \ref{prop3.2},
\ref{prop3.3}, respectively. In view of Propositions~\ref{prop3.4}, \ref{prop3.5},
we state the following conjecture.

\begin{conjecture}
If $G$ is a finite nonabelian simple group, then there exist
two distinct primes $p,q\in oe(G)$, such that,
for all $x,y\in G$ with
$|x|=p$, $|y|= q$, the subgroup $\gen{x,y}$ is nonsolvable (or, maybe, even nonabelian
simple).
\end{conjecture}

The authors are grateful to Frank L\"ubeck, Pham Tiep and Thomas
Wiegel for supplying us with very important information concerning
simple groups of Lie type. We are also grateful to Bob Guralnick
and Gunter Malle for conveying to us results from their paper
\cite{GM} prior to its publication.

\subsection{Other generalisations of Thompson's theorem}

Several oth\-er `Thompson-like' results have appeared in the literature recently. We mention here four such theorems. 
In the first three results, solvability of all $2$-generator subgroups is replaced by a weaker condition restricting 
the required set of solvable $2$-generator subgroups, in different ways from our generalisation.

In 2000, Guralnick and Wilson \cite{GW} obtained a solvability criterion by restricting the proportion of 
$2$-generator subgroups required to be solvable. 

\begin{theorem}
A finite group is solvable if and only if more than $\frac {11}{30}$ of the pairs of elements of $G$      
generate a solvable subgroup.
\end{theorem}

In addition they proved similar results showing that the properties of nilpotency and having odd order are 
also guaranteed if a sufficient proportion of element pairs generate subgroups with these properties, 
namely more than $\frac 12$ for nilpotency, and more than $\frac {11}{30}$ for having odd order.

In contrast to this, in a paper published in 2009, Gordeev, Grune\-wald, Kunyavski\v i and Plotkin \cite{GGKP} 
proved a solvability criterion which involved $2$-generation within each conjugacy class.
This result was also proved independently   by Guest in \cite[Corollary 1]{G}.

\begin{theorem}
A finite group $G$ is solvable if and only if, for each conjugacy class $C$ of $G$, each pair of elements of $C$ 
generates a solvable subgroup.
\end{theorem}
 
A stronger result of this type was obtained recently by Kaplan and Levy in \cite[Theorem 4]{KL}. 
Their criterion involves
only a limited $2$-generation within the conjugacy classes of elements of odd prime-power order.

\begin{theorem}
A finite group $G$ is solvable if and only if for every $x,y\in G$, with $x$ a $p$-element for some odd
prime $p$ and $y$ a $2$-element, the group $\gen{x,x^y}$ is solvable.
\end{theorem}

Our requirement, while ranging over all conjugacy classes, requires only \emph{existence} of a solvable 2-generator 
subgroup with one generator from each of two (possibly equal) classes. We know of no similar criteria in this respect.

The forth result we draw attention to is in a 2006 paper of Guralnick, Kunyavski, Plotkin and Shalev \cite{GKPS}. 
They proved that membership of the solvable radical of a finite group is characterised by solvability 
of certain $2$-generator subgroups. (The solvable radical $R(G)$ of a finite group $G$ is the largest 
solvable normal subgroup of $G$.)

\begin{theorem}
For a finite group $G$, the solvable radical $R(G)$ coincides with the set of all elements $x\in G$ with the property: ``
for any $y\in G$, the subgroup $\gen{ x,y}$ is solvable".
\end{theorem}

\section{Alternating and sporadic simple groups}
\label{sec:altspor}

Theorem~B for the alternating groups follows from the following proposition.

\begin{proposition}\label{prop3.4}
If $m \geq 5$, then there exist distinct
primes $p$ and $q$ satisfying
$m/2< p< q  \leq  m$ such that,
for all $x,y\in A_m$ with $|x|=p$ and $|y|=b$, the subgroup
$\langle x,y\rangle \cong A_d$ for some  $d\geq5$.
In particular, $\langle x,y\rangle $ is nonsolvable.
\end{proposition}

\begin{proof}
First we remark that if $n$ is a positive integer and
$\pi(n)$ denotes the number of primes at most $n$, then the following
is known (see, for example \cite[p.188]{R}):
$$
\pi(2n) - \pi(n) > n/(3 \ln n)\qquad\text{for}\ \ n \geq 5.$$
In particular, if $n \geq  17$, then
$$
\pi(2n) - \pi(n) >17/(3 \ln 17) > 2,
$$
which implies that $\pi(2n) - \pi(n) \geq 3$.
Thus, if $n \geq 17$, then there are at least
$3$ distinct primes $p$, $r$ and $q$
satisfying $n<p<r<q\leq 2n$.
In particular, $n< p < q-3 <q\leq 2n$.
Hence, if $m \geq 34$,
then there exist primes $p$ and $q$
such that $m/2 < p < q-3  <q\leq m$, and elements
$x, y\in A_m$ of order $p$ and $q$, respectively. Let
$x,y$ be any such elements and let $H=\langle x,y\rangle$.
Let  $\Delta$ be the support of $H$, that is, 
the subset of $\{1,2,\dots,m\}$ consisting 
of all the points moved by $H$, and let $d:=|\Delta|$.
Since $q>p>m/2$, it follows that $H$ is transitive 
on $\Delta$, and that  $q\leq d \leq m <2p$.
By \cite[Theorem 8.4]{Wi},
$H$ is primitive on $\Delta$. Moreover, since
$d \geq q > p+3>5$  and $H$ contains only
even permutations, it follows from a theorem of C. Jordan dating from 1873, see 
\cite[Theorem 13.9]{Wi}, that $H\cong A_d$
for some $d\geq 5$, as claimed.

It remains to deal with $A_m$,
for $5\leq m \leq 33$. In each case we will choose primes $p$ and $q$ 
such that $m/2\leq 
p<q\leq m$, and consider the subgroup $H= \langle x,y\rangle$
generated by elements $x$ and $y$ of $A_m$ of 
order $p$ and $q$, respectively. Denote by  $\Delta$  the support of $H$, and
let $d = |\Delta|$. In all cases $d\geq5$ since $d\geq q>p\geq 3$, 
and hence $A_d$ is nonsolvable.

For $17\leq m\leq 33$, and for $11\leq m\leq 13$, let $q$ be the largest 
prime such that $q\leq m$, and let $p$ be the smallest prime such that
$p>m/2$. Then $p\leq q-3$, and the argument above shows that $H\cong A_d$.

If $14\leq m\leq 16$, let $q=13$, $p=11$. If $d=13$, then 
$H \leq A_{13}$ and since by the \cite[p. 104]{ATLAS} no maximal 
subgroup of $A_{13}$ has order divisible by $11\cdot 13$, it 
follows that  $H=A_{13}$. If $d>13$, then as before, $H$ is a primitive group
on $\Delta$, and since $d=p+k$ with $k\geq 3$, it
follows by \cite[Theorem 13.9]{Wi} that $H\cong A_d$.

If $7\leq m \leq 10$, let $q=7$, $p=5$. It follows from the
lists of maximal subgroups of $A_d$ in \cite[pp. 10, 22, 37, 48]{ATLAS} that
$H\cong A_d$. 
If $m = 5,\ 6$, let $q=5$, $p=3$. It follows from the lists
of maximal subgroups of these groups in \cite[pp. 2, 4]{ATLAS} that
$H=A_5$ if $m=5$, and that $H\cong A_5$ or $A_6$ if $m=6$.
The proof of Proposition~\ref{prop3.4} is complete.
\end{proof}

Theorem~B for the sporadic simple groups follows from  the following 
proposition. For compactness of notation, we write $L_2(q)$ instead 
of $\PSL(2,q)$ in Table~\ref{sporad}.

\begin{proposition}\label{prop3.5}
Let $S$ be a sporadic simple group as in one of the rows of 
Table~{\rm\ref{sporad}}. Then for the primes $p, q$ in the corresponding
row of Table~{\rm\ref{sporad}}, 
$p,q \in oe(S)$ and, for all $x,y \in S$ 
with $|x| =p$ and $|y| = q$, the subgroup $\langle x,y\rangle$ is 
one of the nonabelian simple groups in the row for $S$ and column labeled 
$\langle x,y\rangle$ of Table~{\rm\ref{sporad}}.
\end{proposition}

\begin{table}
\begin{center}
\begin{small}
\begin{tabular}[!h]{|rrrl|rrrl|}
\hline
$S$&$p$&$q$&$\langle x,y\rangle$&$S$&$p$&$q$&$\langle x,y\rangle$\\ \hline
$M_{11}$& 2&11& $M_{11},L_2(11)$&$M_{12}$& 2&11& $M_{12},M_{11},L_2(11)$\\
$M_{22}$& 2&11& $M_{22},L_2(11)$&$M_{23}$& 2&23& $M_{23}$\\
$M_{24}$& 2&23& $M_{24},L_2(23)$&$J_1$   & 5&19& $J_1$\\
$J_2$   & 2& 7& $J_2, L_2(7)$   &$J_3$   & 2&19& $J_3, L_2(19)$\\
$J_4$   & 3&43& $J_4$           &$HS$    & 2&11& $HS,M_{11},L_2(11),M_{22}$\\
$He$    & 7&17& $He$            &$McL$   & 2&11& $McL,M_{11},M_{12},L_2(11)$\\
$Suz$   &11&13& $Suz$           &$Ly$    & 3&67& $Ly$\\
$Ru$    & 3&29& $Ru$            &$O'N$   & 2&31& $O'N$\\
$Co_1$  &13&23& $Co_1$          &$Co_2$  & 2&23& $ Co_2,M_{23}$\\
$Co_3$  & 2&23& $Co_3,M_{23}$   &$Fi_{22}$&11&13& $Fi_{22}$\\
$Fi_{23}$ & 2&23& $Fi_{23},M_{23},L_2(23)$&$Fi'_{24}$& 3&29& $Fi'_{24}$\\
$HN$      & 5&19& $HN$                    &$Th$ &19&31&$Th$\\
$B$     & 2&47  &$B$                      &$M$  & 2&59&$M,L_2(59)$\\
\hline
\end{tabular}
\end{small}
\caption{Results table for Proposition~\ref{prop3.5}}\label{sporad}  
\end{center}
\end{table}

\begin{proof}
The proof uses heavily the lists of maximal subgroups of the sporadic
simple groups and of other simple groups, which appear  in \cite{ATLAS} 
and in \cite{ATLAS3}.
For all sporadic simple groups  
our results follow from a close examination of these lists. In particular
the groups listed in the column labeled $\langle x,y\rangle$ are 
the only subgroups which could possibly 
be generated by two elements, the first of order $p$ 
and the second of order $q$.

As an example, we describe in detail our treatment of the Higman--Sims 
sporadic group $HS$ of order $2^9.3^2.5^3.7.11$.
Having checked the maximal subgroups of $HS$, we choose the primes
$p=2$ and $q=11$. If the subgroup $X=\langle x,y\rangle$ is not equal to 
$HS$, then $X$ is contained in a maximal subgroup of $HS$ of order 
divisible by $11$. We see from \cite[p.\,80]{ATLAS} that the only such maximal
subgroups of $HS$ are the simple groups $M_{22}$ and $M_{11}$.
Suppose first that $X$ is a subgroup of $M_{11}$. If $X\neq M_{11}$,
then it is contained in a maximal subgroup of $M_{11}$ of order
 divisible by $11$. By \cite[p.\,18]{ATLAS}, each
such maximal subgroup of $M_{11}$ is isomorphic to the simple group $L_2(11)$. 
If $X\neq L_2(11)$, then $X$ is contained
in a maximal subgroup of $L_2(11)$ of order divisible by $11$.
However,  such maximal subgroups have order $5\cdot 11$, which 
is not divisible by $2$. Thus we are
left with the possibilities: $X=HS$, $X=M_{11}$, $X=L_2(11)$, or 
$X$ is a subgroup of $M_{22}$. If $X$ is a proper subgroup of $M_{22}$, 
then $X$ is contained in a maximal subgroup
of $M_{22}$ of order divisible by $11$. By \cite[p.\,30]{ATLAS}, 
the only such maximal subgroup of $M_{22}$ is the simple group $L_2(11)$, 
which we have already examined. Thus, finally, $X$ is one of the
simple groups $HS$, $M_{11}$, $L_2(11)$ or $M_{22}$, as in 
Table~\ref{sporad}. 
\par
Thus Proposition~\ref{prop3.5} is proved.
\end{proof}

\section{Primitive prime divisors}
In the following,  $q = p^k$ is a power of a prime $p$.
For any positive integer $e$, we say that a prime $r$ is a
\emph{primitive prime divisor} of $q^e -1$ if $r$ divides $q^e -1$
and $r$ does not divide $q^i -1$ for any positive integer $i < e$.
Observe that then $e$ is the order of $q$  modulo the prime $r$; so  $e$ divides $r-1$ and, in particular, $r \geq e+1$.
The set of primitive prime divisors of $q^e -1$
will be denoted by $\pd q e$.

We say that a prime $r$ is a
\emph{basic} primitive prime divisor of $q^e -1$, if $r$ is a primitive
prime divisor of $p^{ke} -1$, that is to say, if $r\in\pd p {ek}$.
We denote by $\bpd q e$ the set of
basic primitive prime divisors of $q^e -1$.
Note that $\bpd qe \subseteq \pd qe$, and that the inclusion can be strict;
for example, $\pd {2^2}3 = \{7\}$, but $\bpd {2^2}3 = \pd 2 6 = \emptyset$.
The following result of Zsigmondy \cite{Z} will be used frequently.
\begin{theorem}\label{Z}
Let $q\geq 2$ and $e\geq 2$. Then $\bpd q e =\emptyset$ if and only if one of the following holds.
\begin{description}
\item [(i)] $q$ is a Mersenne prime, $e=2$, and here $\pd q 2 =\emptyset$;
\item[(ii)]  $(q,e)=(2,6)$, and here $\pd 2 6  =\emptyset$;
\item [(iii)] $(q,e)\in \{(4,3),(8,2)\}$, and here $\pd 4 3 =\{7\}$, and\\ $\pd 8 2 =\{3\}$.
\end{description}
\end{theorem}

Next, we define the set $\lpd qe$ of \emph{large primitive divisors}
of $q^e-1$ to be the set consisting of primes $r\in \pd qe$ such
that $r>e+1$ together with the integer $(e+1)^2$ if $e+1\in\pd q e$ and $(e+1)^2$ divides $q^e-1$.
\par
Finally, we define the set $\lbpd qe$ of \emph{large basic primitive divisors}
of $q^e-1$ to be the set consisting of primes $r\in \bpd qe$ such
that $r>e+1$, together with the integer $(e+1)^2$ if $e+1\in\bpd qe$ and $(e+1)^2$ divides $q^e-1$.
\medskip
\par
The following observation will be very useful in the sequel.
\begin{proposition}\label{Lbpd}
Let $q\geq 2$ and $e\geq 3$. Then $\lbpd qe=\emptyset $ if and only if $(q,e)$ is one of the following:
$$
(2,4),(2,6),(2,10),(2,12),(2,18),(3,4),(3,6),(4,3),(5,6).
$$
\end{proposition}

\begin{proof} 
Let $\mathrm{Lpd}, \mathrm{Lbpd}, \mathrm{bppd}, \mathrm{ppd}$ denote the sets 
$\lpd qe$, $\lbpd qe$, $\bpd qe$, $\pd qe$, respectively.
We prove first that $\mathrm{Lbpd}  \neq \emptyset$ if and only if both $\mathrm{bppd}  \neq\emptyset$ and 
$\mathrm{Lpd}  \neq \emptyset$.  If $\mathrm{Lbpd}\neq \emptyset$,
then clearly $\mathrm{bppd}\neq \emptyset $ and $\mathrm{Lpd}\neq \emptyset$. Conversely,
suppose that $\mathrm{bppd}\neq \emptyset$ and $\mathrm{Lpd}\neq \emptyset$. Let $r$ be the 
largest element of $\mathrm{bppd}$. Then $r\in \pd p {ke}$, so $r\geq ke+1\geq e+1$. If $r>e+1$, then
$r\in \mathrm{Lbpd}$, so we may assume that $r=e+1$. Then $k=1$ and $\mathrm{bppd}=\mathrm{ppd}=\{e+1\}$.
Let $s\in \mathrm{Lpd}$. Since $\mathrm{ppd}=\{e+1\}$, it follows that $s=(e+1)^2$ and $s$ divides $q^e-1$, 
whence $s\in \mathrm{Lbpd}$. Thus in both cases $\mathrm{Lbpd}$ is nonempty, as required.
\par
Hence $\mathrm{Lbpd}$ is empty if and only if either $\mathrm{Lpd}$ is empty or 
$\mathrm{bppd}$ is empty. Now assume that $q\geq 2 $ and $e\geq 3$.
By \cite[Theorem 2.2]{NP1}, the set $\mathrm{Lpd}$
is empty only for $q=2$ and $e\in \{4,6,10,12,18\}$, for $q=3$
and $e\in \{4,6\}$, and for $(q,e)=(5,6)$. Moreover, by Theorem~\ref{Z},
the set $\mathrm{bppd}$ is empty only if 
$(q,e)=\{(2,6),(4,3)\}$. So $\mathrm{Lbpd}$ is empty precisely for the values
of $(q,e)$ listed in the proposition.
\end{proof}

\begin{remark}
\label{Rem 2.3}
\rm {Let $G$ be a subgroup of $\GL(d,q)$ and let $m\in\lpd qe$ 
(or $m\in\lbpd qe$) with $d\geq e > d/2$. If $m$ divides
$|G|$, then $m\in oe(G)$. In fact, either $m=r$ or $m=r^2=(e+1)^2$,
where $r$ is a (basic) primitive prime divisor of $q^e-1$. Since 
$e>d/2$, a Sylow $r$-subgroup of $\GL(d,q)$ is
cyclic and so $G$ has elements of order $m$.}
\end{remark}

We say that an element $g\in G$ is a \emph{$\bpd q e$-element} if the order of $g$
is divisible by some element of $\bpd qe$. Similarly,
we say that $g\in G$ is an \emph{$\lbpd q e$-element} if the order of $g$ is divisible
by some element of $\lbpd qe$.
\par
We use results from ~\cite{NP1} and \cite{NP2} to deal with subgroups
of linear groups containing one or two ``big" ppd-elements. We observe here that
{\it basic} primitive prime divisors will be relevant in order to exclude examples of
``subfield type" (see case (d) in the proof of Lemma~\ref{lem2.4}), while {\it large} 
primitive divisors will
be relevant in the proof of Lemma~\ref{lem3.1}. Our first lemma deals with irreducible subgroups of $\GL(d,q)$ for $d\geq 3$.

\begin{lemma}\label{lem2.4}
Let $G$ be an irreducible subgroup of $\GL(d,q)$, with $d\geq 3$. Assume that 
$x,y\in G$ are such that $x$ is a $\bpd q e$-element and $y$ is a $\bpd q f$-element,
with $d\geq e>f>d/2$.  
Then either $G$ is nonsolvable or one of the following holds.
\begin{description}
\item[(1)] $d=e=f+1$ is prime, $G$ is conjugate to a subgroup of $\GL(1,q^d).d$, and $G$ contains no $\lbpd q f$-elements;
\item[(2)] There exists a prime $c$ dividing $\gcd (d,e,f)$ such that $G$ is conjugate to a subgroup of $\GL(\frac{d}{c},q^c).c$
and  the elements $x^c$ and $y^c$ lie in $\GL(\frac{d}{c},q^c)$ as a $\bpd {q^c}{\frac{e}{c}}$-element   and  
a $\bpd {q^c}{\frac{f}{c}}$-element, respectively.
\end{description}
\end{lemma}
\begin{proof}
By assumption $|x|$ is a multiple of an element $r\in \bpd q e$ and  $|y|$ is a multiple of an element $s\in \bpd q f$. 
Suppose first that $d=3$ and $q\leq4$. Then $3\geq e>f>3/2$ and so $e=3, f=2$ and $\bpd q3\ne\emptyset, \bpd q2\ne\emptyset$. 
By Theorem~\ref{Z}, $q\ne 3,4$. Hence, $q=2$ and so $r=7, s=3$, $G\leq \GL(3,2)$ and $\lbpd q f = \emptyset$. 
By \cite[p.3]{ATLAS}, 
either $G=\GL(3,2)$ is nonsolvable or $G=\gen {x, y}\cong \GL(1,8).3$ and part (1) holds. 
Thus if $d=3$ we may assume that $q\geq 5$.
\par
Next we claim that, if $(d,q)=(4,2)$ or $(4,3)$, then $G$ is nonsolvable.
In these cases  $d=4\geq e>f>2$ and so $e=4, f=3$ and $\bpd q4\ne\emptyset$, $\bpd q3\ne\emptyset$. 
If $(d,q)=(4,3)$, then $G\leq \GL(4,3)\cong 2.\PSL(4,3).2$ (see \cite[pp. 68 and 69]{ATLAS})
and $r=5$, $s=13$. 
Since no proper subgroup of $\PSL(4,3)$ is divisible by $5\cdot 13$, $\PSL(4,3)$ is a section of $G$ and hence $G$ is nonsolvable,
as claimed. If $(d,q)=(4,2)$, then $G\leq \GL(4,2)\cong \PSL(4,2)$ and 
$r=5, s=7$. It follows from  \cite[pp. 10 and 22]{ATLAS} that $G\cong A_7$ or $A_8$ and in particular $G$ is
nonsolvable. Thus we may assume further that $(d,q)\ne(4,2)$ or $(4,3)$.

Note that $q=p^k$ for some prime $p$ and $k\geq1$. Since $r\in\pd p{ek}$, $s\in\pd p{fk}$ 
and $e>f\geq2$, neither $r$ nor $s$ divides $q-1$. Moreover, 
if $i\leq d$ and $h\leq \frac{k}{2}$, then neither $r$ nor $s$ divides $p^{ih}-1$ (since $ih\leq \frac{dk}{2} < fk<ek$). 

We apply \cite[Theorem 4.7]{NP1}.
First we show that the cases (a), (c), (d)
and (e) of \cite[Theorem 4.7]{NP1}, if they occur, imply that $G$ is nonsolvable.
\par
(a): (Classical type) $G$ has a normal subgroup $\Omega$, where $\Omega $ is one of 
the following groups: $\SL(d,q),\ \Sp(d,q)\ (d\ \text{even}),\ \SU(d,q_0)\ (q=q_0^2),\ \Omega ^\pm(d,q)\ 
(d\ \text{even}),\ \text{and}\ \Omega ^\circ(d,q)\ (d\ \text{odd})\ .$
Because of our assumption that $(d,q)\ne (4,2), (4,3)$ or $(3,q)$ with $q\leq 4$, 
each of these classical groups is nonsolvable and so $G$ is nonsolvable in case (a).
\par
(c): (Nearly Simple Groups) Here $G$ is nearly simple, and hence nonsolvable.
\par
(d): (Subfield type) Denote by $Z$ the subgroup of scalar matrices of $\GL(d,q)$
and let $Z\circ \GL(d,q_0)$ denote a group which can be defined over a subfield
modulo scalars. In case (d), $G$ is conjugate to a subgroup of $Z\circ \GL(d,q_0)$ for some 
\emph{proper} subfield $\mathbb F_{q_0}$ of $\mathbb F_q$, say $q=q_0^a $ and $q_0=p^h$, with $k=ah$ and $a\geq 2$. 
Since $x\in G$ and $r$ does not divide $q-1$ (as noted above),
it follows that $r$ divides $|\GL(d,q_0)|$. Hence $r$ divides $p^{ht}-1$ for some $t\leq d$. Since $h\leq \frac{k}{2}$, 
this contradicts our observation above, so case (d) does not arise.
\par
(e): (Imprimitive type) Here $G$ preserves a direct-sum decomposition $V=U_1\oplus
\dots\oplus U_d$ with $\dim(U_i)=1$ for $i=1,\dots,d$ and $G$ induces $A_d$ or $S_d$
on the set $\{U_1,\dots,U_d\}$. Since $A_d$ is nonsolvable for $d\geq 5$, we may 
assume that $d\leq4$. In particular, $|G|$ divides $(q-1)^dd!$ and since 
the primes $r, s$ do not divide $q-1$, each of them divides $d!$.
In particular, $r,s\leq d\leq 4$. However, $e>f\geq 2$, so
$r\geq e+1\geq 4$, which implies that $r=4$, a contradiction.
\par
This leaves case (b) of \cite[Theorem 4.7]{NP1}, and in that case $G$ is conjugate to a subgroup of 
$\GL(\frac{d}{c},q^c).c$, for some prime $c$ dividing $d$. Moreover either $c=d=f+1=e$, or $c<d$ and 
$c$ divides $\gcd (d,e,f)$. In the former case, $G\leq \GL(1,q^d).d$ and $s=d=f+1$. As $s^2$ does not divide $|G|$, 
part (1) holds. 
In the latter case, since each of $r, s$ is at least $f+1> c$, it follows that $x^c, y^c\in\GL(\frac{d}{c},q^c)$ 
and have orders divisible by $r, s$, respectively. Thus part (2) holds.
\end{proof}

If $V$ is a $G$-module and $U$ is a $G$-invariant section of $V$ (that is to say, $U=V_1/V_2$
where $V_2\leq V_1$ are $G$-submodules of $V$), we denote by $G^U$ the group of 
automorphisms induced by $G$ on $U$. So $G^U\cong G/C_G(U)$, a factor group of
$G$, and accordingly we denote by $x^U$ the image in $G^U$ of an element $x\in G$
under the relevant projection homomorphism.

\medskip
For a prime $r$ and integer $n$, let $n_r$ denote the \emph{$r$-part} of $n$, that is, the largest power of $r$ dividing $n$.

\begin{remark} 
\label{Rem 2.5}
\rm {Let $G\leq \GL(d,q)$ and let $V=V(d,q)$ be the natural
module for $G$. Suppose that $x$ is a $\ppd(q,e)$-element of $G$ of order divisible by a primitive prime divisor 
$r$ of $q^e-1$ and let $x_0$ 
be an element of order $r$ in $\langle x\rangle$.  Then the following statements hold.
\par 
(1) Under the action of $\langle x_0\rangle$, $V|_{\langle x_0\rangle}$ is a completely reducible 
$\mathbb F_q\langle x_0\rangle$-module, by Maschke's Theorem, and all the nontrivial 
irreducible submodules of $V|_{\langle x_0\rangle}$ have dimension $e$ 
(see for instance \cite[Theorem II.3.10]{H}). In particular, it follows that there 
exists at least one $G$-composition factor $U$ of $V$ on which $x_0$,
and hence also $x$,  acts nontrivially.
So $|x_0^U|=r$ and hence $x^U$ is a $\ppd(q,e)$-element of
$G^U\leq \GL(d_0,q)$ of order divisible by $|x_0^U|=r$ and  satisfying $|x^U|_r=|x|_r$, 
where $d_0=\dim_{\mathbb F_q}(U)\geq e$.
\par
(2) Suppose also that $y$ is a $\ppd(q,f)$-element of $G$ of order divisible by a primitive prime divisor $s$ of 
$q^f-1$, and that $d\geq e>f>d/2$.
Since both $e$ and $f$ are greater than $d/2$, it follows by part (1) of this remark that
there exists a unique $G$-composition factor $U$ of $V$,
of dimension $d_0=\dim_{\mathbb F_q}(U)\geq e$, on which
both $x$ and $y$ act nontrivially. Moreover $G^U\leq \GL(d_0,q)$, and 
$x^U$ is a $\ppd(q,e)$-element of $G^U$ of order satisfying $|x^U|_r=|x|_r\geq r$
and $y^U$ is a $\ppd(q,f)$-element of $G^U$ of order satisfying $|y^U|_s=|y|_s\geq s$.}
\end{remark}
\medskip
We now apply Lemma~\ref{lem2.4} to linear groups of dimension large enough to allow the
existence of primitive prime divisors for two distinct exponents, both
greater than half the dimension of the linear group. It is convenient 
to deal separately  with certain large exponent pairs.

\begin{lemma}\label{lem2.6}
Let $j$ be a nonnegative integer, $\delta\in\{1,2\}$, and let $x,y\in \GL(d,q)$ be such that
$x$ is a $\bbppd(q,e)$-element and $y$ is a $\bbppd(q,f)$-element.
Assume that $e$, $f$, $\delta$, and $d$ satisfy  the following conditions:
\[
e=d-j,\ f=d-j-\delta,\ d\geq 2j+2\delta +1. 
\]
If $\delta=1$, we assume in addition that $y$ is an $\mathrm{Lbpd}(q,f)$-element.
Then $G=\langle x,y\rangle$ is nonsolvable.
\end{lemma} 
\begin{proof}
Let $r\in\bpd q e$ such that $r$ divides $|x|$, and let $s\in\bpd qf$ such that $s$ divides $|y|$. 
As mentioned above, $r\equiv 1\pmod e$ and
$s\equiv 1\pmod f$. Moreover, $e>f=d-j-\delta\geq  
j+\delta +1\geq 2$, so in particular both $r$
and $s$ are odd primes. As it is sufficient to prove that some subgroup
of $G$ is nonsolvable, we may assume that $G=\langle x,y\rangle$, and 
$|x|,|y|$ are powers of $r$ and $s$, respectively. In particular, $|x|$ and $|y |$ are odd. 
If $\delta=1$, then in addition we assume that $y$ is an  
$\mathrm{Lbpd}(q,f)$-element.
\par
By our assumptions, $j\leq \frac {d-1}2-\delta$, so $d\geq e>f=d-j-\delta \geq \frac {d+1}2>\frac d2$.
In the following, we denote by $V$ the natural module $V(d,q)$ for $\GL(d,q)$.
\par
{\it Case (1): $G$ is irreducible on $V$.}
\par
The conditions of Lemma~\ref{lem2.4} are satisfied. Applying that result we deduce that
either $G$ is nonsolvable, or (1) $d=e=f+1$ is prime, $G$ is conjugate to a subgroup of $\GL(1,q^d).d$, 
and $G$ contains no $\mathrm{Lbpd}(q,f)$-elements, or (2) there is a prime $c<d$  such that $c$ 
divides $\gcd(d,e,f)$, $G$
is conjugate to a subgroup of $\GL(d/c,q^c).c$, and  $x^c$ and $y^c$ lie in $\GL(\frac{d}{c},q^c)$ 
as a $\bpd {q^c}{\frac{e}{c}}$-element   and  a $\bpd {q^c}{\frac{f}{c}}$-element, respectively.
If (1) holds, then $j=0, \delta=1$ and we have a contradiction, since $y$ is an
$\mathrm{Lbpd}(q,f)$-element of $G$. Thus  (2) holds. Since $c$ is a prime
dividing $\gcd(d,e,f)$, and since $\gcd(e,f)=\gcd(e,e-f)\leq e-f=\delta\leq 2$, it follows that 
$c=\delta=2$ and all of $d,e,f,j$ are even. Hence $G\leq \GL(d/2,q^2).2$, and since the orders of $x,y$ 
are odd, we conclude that $G\leq \GL(d/2,q^2)$. Since $d$ and $j$ are even, 
our assumption that $d\geq 2j+2\delta+1=2j+5$ implies $\frac d2\geq 2\frac j2+3$. Thus replacing $(d,q,e,f,j,\delta)$
by $(\frac d2,q^2,\frac e2,\frac f2,\frac j2,\frac \delta 2)$, all the conditions of the lemma hold, 
and $G$ is irreducible on $V(\frac d 2, q^2)$.
By the arguments above and since $\frac \delta 2 =1$, we conclude that $G$ must be
nonsolvable. 
\par
{\it Case (2): $G$ is reducible on $V$.}
\par
If $e=d$, then $|x|$ would be a multiple of a primitive prime divisor 
of $q^d-1$ and so $\langle x\rangle$ would act irreducibly on the 
natural module $V$. However $G$ is reducible, so we must have $e<d$ 
and  $j\geq 1$. By Remark \ref{Rem 2.5}(2), there exists a unique 
$G$-composition factor $U$ of $V$ of dimension $d_0=\dim_{\mathbb F_q}(U)$,
say, where $d>d_0\geq e=d-j$, such that $x^U$ and $y^U$ are $\bbppd(q,e)$-\ 
and $\bbppd(q,f)$-elements of $G^U$, respectively, with $|x^U|_r=|x|_r\geq r$, 
$|y^U|_s=|y|_s\geq s$.  In particular, if $y$ is an 
$\mathrm{Lbpd}(q,f)$-element of $G$, then $y^U$ is an 
$\mathrm{Lbpd}(q,f)$-element of $G^U$. 

We claim that the irreducible group $G^U=\langle x^U, 
y^U\rangle\leq\GL(d_0,q)$ 
induced by $G$ on $U$ satisfies the conditions of Lemma~\ref{lem2.6} 
with parameters $d_0,e,f,\delta$, relative to the integer $j_0:=j-d+d_0$. 
Note that $j>j_0=d_0-e\geq0$ and $d_0=j_0+d-j\geq j_0+j+2\delta+1 
>2j_0+2\delta+1\geq 3$, and the conditions $e=d_0-j_0, f=d_0-j_0-\delta$ hold, by 
the definition of $j_0$. This proves the claim. 
Since $G^U $ is irreducible, it follows from Case (1) 
of this proof that $G^U$ is nonsolvable. Consequently, also 
$G$ is nonsolvable.
\end{proof} 

We can now prove Theorem~B for classical simple groups of large dimension.

\begin{proposition}\label{prop2.7}
Assume that $S$ is one of the following simple groups:
\begin{description}
\item[(1)] $\PSL(d,q)$, with $d\geq 4$ and 
$(d,q)\neq (6,2)$;
\item[(2)] $\PSp(d,q)$ ($d$ even) or $\PSU(d,q)$ ($d$ odd), with $d\geq 5$ and 
\newline $(d,q)\notin \{(5,2),(6,2),(8,2)\}$;
\item[(3)] $\POm^{\circ}(d,q)$ ($dq$ odd) or $\PSU(d,q)$ ($d$ even), with $d\geq 7$;
\item[(4)] $\POm^{\pm}(d,q)$ ($d$ even), with $d\geq 10$ and $(d,q)\neq (10,2)$.
\end{description}
Then there exist $a,b\in oe(S)$ such that for every choice of $x,y\in S$ with
$|x|=a$ and $|y|=b$, the group $\langle x,y\rangle$ is nonsolvable.
\end{proposition}
\begin{proof}
We work with the group $\hat S$ defined as follows:
\newline $\bullet$ $\hat S=\SL(d,q)\leq \GL(d,q)$, when $S=\PSL(d,q)$;
\newline $\bullet$ $\hat S=\Sp(d,q)\leq \GL(d,q)$, when $S=\PSp(d,q)$;
\newline $\bullet$ $\hat S=\SU(d,q)\leq \GL(d,q^2)$, when $S=\PSU(d,q)$;
\newline $\bullet$ $\hat S=\Omega^\varepsilon(d,q)\leq \GL(d,q)$, 
when $S=\POm^\varepsilon(d,q)$, for $\varepsilon=\pm$ or $\circ$.
\newline Let $\mathbb Z=\mathbb Z(\hat S)$. 
We prove that there exist $a,b\in oe(\hat S)$ with 
$\gcd(ab,|\mathbb Z|)\newline
=1$ such that, for every $\hat x\in \hat S$ of order
$a$ and every $\hat y\in \hat S$ of order
$b$, the group $\langle \hat x,\hat y\rangle$ is nonsolvable. Since 
$a$ and $b$ are coprime to $|\mathbb Z|$, if $|\hat x|=a$ then also 
the order of $\hat x\mathbb Z$ in $S$ is equal to $a$, and similarly, if
$|\hat y|=b$ then $|\hat y\mathbb Z|=b$. 
In addition, if $\langle\hat x, \hat y\rangle$ is nonsolvable then also  
$\langle x,y\rangle$ is nonsolvable,  because it is a central 
factor of the nonsolvable group $\langle \hat x,\hat y\rangle$. 
Thus it is sufficient to work
as described with the group $\hat S$. 
In fact, each of  $a$ and $b$ will be either a prime or the square of 
a prime, and it will be easy to check that $\gcd(ab,|\mathbb Z|)=1$.
\medskip
\par
{\bf(1):} Let $\hat S=\SL(d,q)$, with $d\geq 4$. Suppose first  that
\begin{multline}\label{eq2}
$$
(d,q)\notin \{(4,4),(5,2),(5,3),(6,2),(7,2),(7,3),(7,5),\\(11,2),(13,2),(19,2)\}.\tag{1.1}
$$
\end{multline}
Then by Theorem~\ref{Z} there exists $a\in \bpd qd$, and by 
Proposition~\ref{Lbpd} there exists $b\in \lbpd q{d-1}$. 
Now $q^d-1$ and $q^{d-1}-1$ divide $\hat S$
and hence $a,b\in oe(\hat S)$ by Remark~\ref{Rem 2.3}. 
Also $\gcd(ab,|\mathbb Z|)=1$. Let now $\hat x,\hat y \in \hat S$ 
be such that $|\hat x|=a$ and $|\hat y|=b$. Then $\hat x$ is a 
$\bbppd(q,d)$-element of $\hat S$ and $\hat y$ is an 
$\Lbpd(q,d-1)$-element of $\hat S$. Since $d\geq 4$, 
the conditions of  Lemma~\ref{lem2.6} 
are satisfied with $\delta=1, j=0$, and we conclude that 
$\langle \hat x,\hat y\rangle$ is nonsolvable.
\par
We now consider the ``isolated" cases left out of this proof and 
listed in (\ref{eq2}).
We deal with all these cases, except the excluded pair 
$(d,q)=(6,2)$, by taking elements $\hat x,\hat y\in \hat S$
with $|\hat x|=a$ and $|\hat y|=b$, 
where $a,b$ are  as in the following table. (Here $q_i$ denotes a prime in 
$\bpd q i$, and, by Theorem~\ref{Z},  for all entries
in the table such a prime $q_i$ exists.)

\begin{small}
$$\begin{array}{|c|ccccccccc|}\hline
(d,q)&(4,4)&(5,2)&(5,3)&(7,2)&(7,3)&(7,5)&(11,2)&(13,2)&(19,2)\\ \hline
a&q_4&q_5&q_5&q_7&q_7&q_7&q_{11}&q_{13}&q_{19}\\
b&q_2&q_3&q_3&q_5&q_5&q_5&q_9&q_{11}&q_{17}\\ \hline
\end{array}
$$
\end{small}

\medskip
Thus $\hat x,\hat y$
satisfy all the conditions of Lemma~\ref{lem2.6} with $\delta=2, j=0$,
and hence $\langle \hat x,\hat y\rangle$ is nonsolvable.
\medskip
\par
{\bf(2):} If  $\hat S=\Sp(d,q)$ with $d\geq 5$ and $d$ even, then the order of 
$\hat S$ is divisible by both $q^d-1$ and $q^{d-2}-1$. By Theorem~\ref{Z}, 
$\bpd q d\neq \emptyset$ and $\bpd q {d-2}\neq \emptyset$, since 
by our assumptions $d>d-2\geq 4$ and $(q,d),(q,d-2)
\neq (2,6)$.  We take
$a\in \bpd q d$ and $b\in \bpd q {d-2}$ and note that $a,b\in oe(\hat S)$.

If $\hat S=\SU(d,q)\leq \GL(d,q^2)$ with $d\geq 5$ and $d$ odd, then the 
order of $\hat S$ is divisible by both $q^d+1$ and $q^{d-2}+1$. 
Here we take $a\in \bpd q {2d}= \bpd {q^2}d$ and $b\in \bpd q{2d-4}
=\bpd {q^2}{d-2}$. Note that both $\bpd q{2d}$
and $\bpd q{2d-4}$ are nonempty, since $(d,q)\neq (5,2)$.
Also $a,b\in oe(\hat S)$.
\par
In both cases we apply Lemma~\ref{lem2.6} with $\delta=2, j=0$. 
It follows that in each of these two cases, $\langle \hat x,\hat y\rangle$
is nonsolvable for every choice of $\hat x,\hat y\in \hat S$ with 
$|\hat x|=a$ and $|\hat y|=b$.
\medskip
\par 
{\bf(3):} If $\hat S=\Omega ^{\circ}(d,q)$, with $d\geq 7$ and $dq$ odd, 
then both $q^{d-1}-1$ and $q^{d-3}-1$ divide the order of $\hat S$ and
each has a basic primitive prime divisor by Theorem~\ref{Z}. 
We take $a\in \bpd q{d-1}$ and $b\in \bpd q{d-3}$ and note that 
$a,b\in oe(\hat S)$.
\par
If $\hat S=\SU(d,q)\leq \GL(d,q^2)$ with $d\geq 7$ and $d$ even, then the 
order of $\hat S$ is divisible by both $q^{d-1}+1$ and $q^{d-3}+1$. 
We take $a\in \bpd q{2d-2}=\bpd {q^2}{d-1}$
and $b\in \bpd q{2d-6}=\bpd {q^2}{d-3}$, noting that both $\bpd {q^2}{d-1}$ 
and $\bpd {q^2}{d-3}$ are nonempty, by Theorem~\ref{Z}, since 
$d-1>d-3\geq 4$ and $q^2>2$. Thus $a,b\in oe(\hat S)$.
\par
In both cases we apply Lemma~\ref{lem2.6} with $\delta=2, j=1$. 
It follows that in each of these two cases, $\langle \hat x,\hat y\rangle$
is nonsolvable for every choice of $\hat x,\hat y\in \hat S$ with 
$|\hat x|=a$ and $|\hat y|=b$.
\medskip
\par
{\bf(4):} Assume, finally, that $\hat S=\Omega ^{\pm}(d,q)$ with 
$d\geq 10$, $d$ even, and $(d,q)\ne (10,2)$.
In this case both $q^{d-2}-1$ and $q^{d-4}-1$ divide the order of $\hat S$,
and each has a basic primitive prime divisor by Theorem~\ref{Z}, since
$d-2>d-4\geq 6$ and $(d,q)\neq (10,2)$. 
We take $a\in \bpd q{d-2}$ and $b\in \bpd q{d-4}$, and note that $a,b\in 
oe(\hat S)$. 
By Lemma~\ref{lem2.6} with $\delta=2, j=2$, we conclude that 
$\langle \hat x,\hat y\rangle$ is nonsolvable for every choice of 
$\hat x,\hat y\in \hat S$ with $|\hat x|=a$ and $|\hat y|=b$.
\end{proof}  

\section{Proof of Theorem~B}

Proposition~\ref{prop2.7} does not cover the classical groups in small 
dimensions, when the conditions of Lemma~\ref{lem2.6} do not hold. 
For most of these cases we use the following result, 
which deals with the case where there 
is at least one $\Lbpd(q,e)$-element in $G\leq \GL(d,q)$, with $e>d/2$.

\begin{lemma}\label{lem3.1}
Let $G$ be an irreducible subgroup of $\GL(d,q)$, with $d\geq 3$. Assume that
$G$ contains an $\Lbpd(q,e)$-element for some integer $e$ with $d\geq e>d/2$.
Assume also that 
$$(d,q)\notin \{(3,2),(3,3),(3,4),(4,2),(4,3)\}.$$
Then either $G$ is
conjugate to a subgroup of $\GL(d/c,q^c).c$ for some prime divisor $c$ 
of $\gcd(d,e)$, or $G$ is nonsolvable.
\end{lemma}

\begin{proof}
This lemma follows from \cite[Theorem 3.1]{NP2}.
Note that the cases (b), (d) and (c)(i)
of that theorem do not occur, because of the assumption that the 
$\bbppd(q,e)$-element is large. 
So either $G$ is of classical type (part (a) of  \cite[Theorem 3.1]{NP2}) 
and hence is nonsolvable because
of the assumptions on $(d,q)$ (see the details in the proof of
Lemma~\ref{lem2.4}, case (a)), or $G$ is of nearly simple type (part (e) of  
\cite[Theorem 3.1]{NP2}) and hence is nonsolvable, 
or $G$ is of extension field type (part (c)(ii) of  \cite[Theorem 3.1]{NP2}),
that is to say, $G$ is conjugate to a subgroup 
of $\GL(d/c,q^c).c$ for some prime divisor $c$ of $\gcd(d,e)$.
\end{proof}
\medskip
\par
We now complete the proof of Theorem~B for classical groups of Lie type.
\begin{proposition}\label{prop3.2}
Theorem~B holds for all classical finite simple groups of Lie type.
\end{proposition}
\begin{proof}
{\bf (1):} \emph{Assume first that $S=\PSL(d,q)$ with $d\geq 2$.}
\par
Suppose first that $d=2$ and $q\leq 7$. Theorem~B holds for the groups 
$\PSL(2,4)\cong \PSL(2,5)\cong A_5$ by Proposition~\ref{prop3.4}. 
Also, for $S=\PSL(2,7)$, the maximal 
subgroups of order divisible by $7$ have order $21$, so 
$\langle x,y\rangle=S$ whenever $x,y\in S$ with $|x|=2$ and $|y|=7$.
\par
Next we consider $S=\PSL(2,q)$ with $q\geq 8$. 
Take $a=(q+1)/k$ and $b=(q-1)/k$,
where $k=\gcd(q-1,2)$. Then $a\geq 5$, $b\geq 4$ and 
$a,b\in oe(S)$ (see Theorems II.8.3 and II.8.4 in [H]). The 
classification of the subgroups of $\PSL(2,q)$ (see Theorem II.8.28 in [H])
implies that $\langle x,y\rangle=S$ whenever $x,y\in S$ with
$|x|=a$ and $|y|=b$.
\par 
Suppose next that $S=\PSL(3,q)$. The group 
$\PSL(3,2)\cong \PSL(2,7)$ has 
been dealt with already. If $S=\PSL(3,q)$ with $q=3$ or $4$, then 
taking $x,y\in S$ with $(|x|,|y|)=(13,2)$ or $(7,5)$, respectively, 
we find that $S=\langle x,y\rangle$ (see \cite[pp.\,13,23]{ATLAS}).
Hence we may assume that $q\geq 5$. By Proposition~\ref{Lbpd}, there 
exists $a\in \lbpd q 3$. If $q=3^k$,
then $k>1$ and we take $b\in \ppd(q,2)$ which, by Theorem~\ref{Z}, 
is nonempty. On the other hand, if $q$ is 
not a power of $3$, then we choose $b=p$ (recall that $q$ is a power of $p$). 
As in the proof of 
Proposition~\ref{prop2.7},
we work with  $\hat S=\SL(3,q)\leq \GL(3,q)$ in order to apply 
Lemma~\ref{lem3.1}. (We shall 
often do this throughout the proof without further reference.) 
Let $\hat x,\hat y\in \hat S$, with $|\hat x|$ a multiple of $a$
and $|\hat y|$ a multiple of $b$, and let $X=\langle \hat x,\hat y\rangle$. 
Since $\hat x\in X$,  $X$ is an irreducible subgroup of $\GL(3,q)$, and
since $\hat y\in X$, $|X|$ does not divide $|\GL(1,q^3).3|$. Hence, by 
Lemma~\ref{lem3.1}, $X$ is nonsolvable. It follows that 
$\langle x,y\rangle$ is nonsolvable for every $x,y\in S$ with $|x|=a$ and 
$|y|=b$.
\par
Thus we may assume that $d\geq4$. For these groups $S$, Theorem~B 
follows from Proposition~\ref{prop2.7}(1), unless
$(d,q)=(6,2)$. 
For $S=\PSL(6,2)\cong \GL(6,2)$, consider $a=31\in \lbpd 2 5$ and
$b=7\in \bpd 2 3$, and note that $a,b\in oe(S)$. Let $x,y\in \GL(6,2)$ 
with $|x|=31$ and $|y|=7$, and
set $X=\langle x,y\rangle $. If $X$ is reducible on $V=V(6,2)$,
then, since $x\in X$, $X$ acts irreducibly on some $X$-composition 
factor $U$ of $V$ of dimension $5$ and $X^U=\langle x^U,y^U\rangle  \leq 
\GL(5,2)$. By Remark~\ref{Rem 2.5}(2), $|x^U|=|x|=31$ and $|y^U|=|y|=7$, 
and hence $X^U$ is nonsolvable by  Lemma~\ref{lem2.6} applied 
with $\delta=2, j=0$. 
Consequently, also $X$ is nonsolvable.
If $X$ is irreducible, then by Lemma~\ref{lem3.1} (applied with $d=6, e=5$),  
$X$ is nonsolvable, since $\gcd(6,5)=1$.
\par
{\bf (2):} \emph{Next let $S=\PSp(d,q)'$ with $d\geq4$ and $d$ even (noting 
that $\PSp(2,q)\cong \PSL(2,q)$ has been dealt with above).}
\par
First consider $d=4$. The result for $\PSp(4,2)'\cong A_6$ follows from
Proposition~\ref{prop3.4}.  If $S=\PSp(4,3)$, we have $5, 9\in oe(S)$ and
no maximal subgroup of $S$ contains elements of both orders $5$ and $9$ (see  
\cite[p.\,26]{ATLAS}). Hence $S=\langle x,y\rangle$ for all $x,y\in S$ 
with $|x|=5$ and $|y|=9$. If $S=\PSp(4,4)$, then $5,17\in oe(S)$ and (see 
\cite[p.\,44]{ATLAS}) the only maximal subgroups of $S$ containing an 
element of order $17$ are of the form $\PSL(2,16):2$, and every subgroup 
of such a group of order divisible by both $17$ and $5$ contains $\PSL(2,16)$. 
Hence $\langle x,y\rangle$ is nonsolvable whenever $x,y\in S$ with
$|x|=17$ and $|y|=5$.
\par
So suppose that $q\geq 5$ and take $a\in\lbpd q4$, which is nonempty by 
Proposition~\ref{Lbpd}, and $b=(q^2-1)/\gcd(2,q-1)$. 
Note that $a,b\in oe(S)$, since $\PSp(2,q^2)
\cong \PSL(2,q^2)$ is isomorphic to a subgroup of $S$.
We consider $\hat S=\Sp(4,q)\leq \GL(4,q)$. Let $\hat x,\hat y\in 
\hat S$, with $|\hat x|$ a multiple of $a$ and 
$|\hat y|$ a multiple of $b$, and let $ X=\langle \hat x,\hat y\rangle$. 
Then $\hat x$ is an  $\Lbpd(q,4)$-element of $X$, and in particular 
$\hat x$ acts irreducibly on $V(4,q)$. Hence $X$ is an irreducible 
subgroup of $\GL(4,q)$. By Lemma~\ref{lem3.1}, 
either $X$ is nonsolvable
or $X$ is (conjugate to) a subgroup of $\GL(2,q^2).2$. Assume the latter. Then 
$\langle \hat x^2,\hat y^2\rangle \leq \GL(2,q^2)$, $|\hat x^2|$ is a 
multiple of the large primitive divisor $a$ of $q^4-1$ and 
$|\hat y^2|$ is a multiple of $b/\gcd(b,2)\geq 6$. Hence, 
by the classification of the subgroups of $\PSL(2,q^2)$ 
(see \cite[Theorem II.8.27]{H}), 
we conclude that $\langle \hat x^2,\hat y^2\rangle \geq  \SL(2,q^2)$ 
and hence, in particular, 
$X=\langle \hat x,\hat y\rangle$ is nonsolvable.
\par
By part (2) of Proposition~\ref{prop2.7}, we are left with 
the following cases: $\PSp(6,2)$  and $\PSp(8,2)$. 
If $S=\PSp(6,2)$, then $15, 7\in oe(S)$, and for all $x,y\in S$
with $|x|=15, |y|=7$, the group $X=\langle x,y\rangle$ is nonsolvable.
This is seen as follows: one checks from \cite[p.\,46]{ATLAS} that each 
maximal subgroup of $S$ of order divisible by $35$ is isomorphic to $S_8$.  
Moreover, a subgroup of $S_8$ generated by two elements, of orders $15$ and 
$7$, contains $A_8$, and in fact is equal to $A_8$.  
Thus, $X=S$ or $X\cong A_8$.
\par
Finally, let $S=\PSp(8,2)\cong \Sp(8,2)<\GL(8,2)$. Then $17, 7\in oe(S)$. By 
\cite[p.\,123]{ATLAS}, each maximal subgroup
of $S$ of order divisible by $17\cdot 7$ is isomorphic to $\POm^-(8,2):2$.
By \cite[p.\,89]{ATLAS}, each  maximal subgroup of $\POm^-(8,2)$ of order 
divisible by $17$ is isomorphic to  $\PSL(2,16):2$, and so contains no 
elements of order $7$. 
Thus $\langle x,y\rangle$ is nonsolvable whenever
$x,y\in S$ with $|x|=17$ and $|y|=7$.
\par
{\bf (3):} \emph{Assume now that $S=\PSU(d,q)$, with $d\geq3$ and $d$ odd.} 
\par By Proposition~\ref{prop2.7},
we need only consider  the cases $S=\PSU(3,q)$ with $q\geq3$ (since $\PSU(3,2)\cong 3^2:Q_8$ is solvable), and $S=\PSU(5,2)$. 
By \cite[pp.\,72-73]{ATLAS}, if $S=\PSU(5,2)$, then $11,15\in oe(S)$ and each maximal subgroup of $S$ of order divisible 
by $11$ is isomorphic to $\PSL(2,11)$ and
contains no elements of order $15$. Thus if $x,y\in S$ with $|x|=11$ and $|y|=15$, then
$\langle x,y\rangle=S$.
\par
Therefore we may assume that $S=\PSU(3,q)$ with $q\geq3$.  
By \cite[pp.\,14,\,34]{ATLAS}, if $S=\PSU(3,q)$ with $q\in\{3,5\}$, then $7,8\in oe(S)$ and each maximal subgroup of $S$ 
of order divisible by $7$ is isomorphic to $\PSL(2,7)$ if $q=3$, and  to $A_7$ if $q=5$, and neither of these groups 
contains an element of order $8$.  Thus if $x,y\in S$ with $|x|=7$ and $|y|=8$, then
$\langle x,y\rangle=S$.
So we may assume that $q\neq 2,3,5$. Then by Proposition~\ref{Z},  
$\Lbpd(q,6)\neq \emptyset$. Let $a\in \Lbpd(q,6)$, and note that $a\in oe(S)$ by 
Remark~\ref{Rem 2.3}. Also let $b=p$ if $p\neq 3$, and $b=(q-1)/2$ if $p=3$ (recall that $q$ is a power 
of $p$), and note that $b\in oe(S)$ since $\PSU(2,q)\cong \PSL(2,q)$ 
is isomorphic to a subgroup of $S$. Now let $\hat S=\SU(3,q)\leq \GL(3,q^2)$, and note 
that $a,b\in oe(\hat S)$
and that $\gcd(ab,|\mathbb Z(\hat S)|)=1$. Consider $\hat x,\hat y\in \hat S$ with $|\hat x|$ a multiple
of $a$ and $|\hat y|$ a multiple of $b$, and let $X=\langle \hat x,\hat y\rangle$. Since
$a$ is a primitive prime divisor of $(q^2)^3 -1$, $X$ is an irreducible subgroup of $\GL(3,q^2)$. Thus
by Lemma~\ref{lem3.1}, either $X$ is nonsolvable or $X$ is a subgroup of $X_0=\GL(1,q^3).3$. Suppose that $X\leq X_0$. 
If $p\neq 3$ then $b=p$ does not divide $|X_0|$. Hence 
$p=3$ and $b=(q-1)/2$. 
However $\hat x\in X_0$ and $|\hat x|$ is a multiple of $a\in \Lbpd(q,6)$, so $|\hat x|$ does not divide $|X_0|$,
a contradiction. Hence $X$ is nonsolvable and, consequently, $\langle x,y\rangle$
is nonsolvable whenever $x,y\in S$ with $|x|=a$ and $|y|=b$.
\par
{\bf (4):} \emph{Assume now that $S=\POm^{\circ}(d,q)$, with $d$ odd and $d\geq 3$.}
\par Since $\POm^{\circ}(2m+1,2^k)\cong \PSp(2m,2^k)$, for all $m$ and $k$, we may assume that $q$ is odd. 
Also since $\POm^{\circ}(3,q)\cong \PSL(2,q)$ and
$\POm^{\circ}(5,q)\cong \PSp(4,q)$, we may assume that $d\geq7$. In this case Theorem~B  follows 
from Proposition~\ref{prop2.7}.
\par
{\bf (5):} \emph{Assume now that $S=\PSU(d,q)$, with $d$ even.}
\par Since $\PSU(2,q)\cong \PSL(2,q)$, we may assume that $d\geq4$, and 
it follows by Proposition 3.7 that we only have to check dimensions $d=4$ and $d=6$.
\par
Let $S=\PSU(4,q)$. Since $\PSU(4,2)\cong \PSp(4,3)$, we may assume that $q\geq 3$.
For $S=\PSU(4,3)$, it follows from \cite[p.\,52-53]{ATLAS} that $7,9\in oe(S)$
and each maximal subgroup of $S$ of
order divisible by $7$ is isomorphic to $\PSL(3,4)$, $\PSU(3,3)$ or $A_7$, and hence contains no elements of order $9$. 
Thus $\langle x,y\rangle=S$ whenever $x,y\in S$ with $|x|=7$ and $|y|=9$.
\par
Thus we may assume that $S=\PSU(4,q)$ with $q\geq 4$. Then, by Theorem~\ref{Z} and Proposition~\ref{Lbpd}, 
both $\Lbpd(q^2,3)$ and $\bbppd(q,4)$
are nonempty. Let $a\in \Lbpd(q^2,3)$ and $b\in \bbppd(q,4)$, and note that, by Remark~\ref{Rem 2.3}, 
$a,b\in oe(S)$ (since $(q^4-1)(q^3+1)$ divides $|S|$). 
Consider $\hat S=\SU(4,q)\leq \GL(4,q^2)$ acting on the natural module  $V=V(4,q^2)$. Observe that 
$\gcd(ab,|\mathbb Z(\hat S)|)=1$ and hence $a,b\in oe(\hat S)$. 
Let $\hat x,\hat y\in \hat S$ with $|\hat x|$ a multiple
of $a$ and $|\hat y|$ a multiple of $b$, and let $X=\langle \hat x,\hat y\rangle$. 
By Remark~\ref{Rem 2.5}, there exists an $X$-composition factor $U$ of $V$ such that 
both $\hat x$ and $\hat y$ act nontrivially on $U$, with $\dim _{\mathbb F_{q^2}}(U)\geq 3$. 
Moreover $\hat x^U$ is an $\Lbpd(q^2,3)$-element and $\hat y^U$ is an $\Lbpd(q^2,2)$-element of $X^U$ (since $b\geq 5$).
Assume first that $X$ acts reducibly on $V$. 
Then $\dim _{\mathbb F_{q^2}}(U)=3$, and hence, by Lemma~\ref{lem2.6} with $\delta=1,j=0$, the group $X^U$ is nonsolvable. Thus 
also $X$ is nonsolvable. Therefore we may assume that $X$ is an irreducible 
subgroup of $\GL(4,q^2)$. Then, by Lemma~\ref{lem3.1}, $X$ is nonsolvable, as
$a\in \Lbpd(q^2,3)$ and the extension field case cannot occur because 
$\gcd(4,3)=1$.
\par
Finally let $S=\PSU(6,q)$. If $S=\PSU(6,2)$, then by \cite[pp.\,39,\,115]{ATLAS}, 
$7,11\in oe(S)$, each maximal
subgroup of $S$ of order divisible by $7\cdot 11$ is isomorphic to $M_{22}$, and $M_{22}$ has no maximal subgroup 
of order divisible by $7\cdot 11$. Hence  
$\langle x,y\rangle=S$ or
$\langle x,y\rangle\cong M_{22}$ whenever $x,y\in S$ with $|x|=7$ and $|y|=11$.
Thus we may assume that $q>2$. Then, by Theorem~\ref{Z} and Proposition~\ref{Lbpd}, both $\Lbpd(q^2,5)$ and $\bbppd(q,6)$
are nonempty. Let $a\in \Lbpd(q^2,5)$ and $b\in \bbppd(q,6)$. Since $(q^5+1)(q^3+1)$
divides $|S|$, it follows by Remark~\ref{Rem 2.3} that $a,b\in oe(S)$. Consider now 
$\hat S=\SU(6,q)\leq \GL(6,q^2)$ acting on the natural module $V=V(6,q^2)$. Since $\gcd(ab,|Z(\hat S)|)=1$, it follows that 
$a,b\in oe(\hat S)$. Let $\hat x,\hat y\in \hat S$ with $|\hat x|$ a multiple
of $a$ and $|\hat y|$ a multiple of $b$, and let $X=\langle \hat x,\hat y\rangle$.
If $X$ is an irreducible subgroup of $\GL(6,q^2)$, then we conclude by 
Lemma~\ref{lem3.1} that $X$ is nonsolvable, since $a\in \Lbpd(q^2,5)$ and the 
extension field case cannot occur because $\gcd(6,5)=1$. So we may assume that $X$ acts 
reducibly on  $V$.
By Remark~\ref{Rem 2.5}(2), there exists an $X$-composition factor $U$ of $V$ such that
both $\hat x$ and $\hat y$ act nontrivially on $U$, with $\dim _{\mathbb F_{q^2}}(U)\geq 5$.
It follows that $\dim _{\mathbb F_{q^2}}(U)=5$. Moreover, $\hat x^U$ is an $\Lbpd(q^2,5)$-element and $\hat y^U$ 
is a $\bbppd(q^2,3)$-element of $X^U$.
Hence by Lemma~\ref{lem2.6} with $\delta=2,j=0$, the group $X^U$ is nonsolvable and so is $X$.
\par   
{\bf (6):} \emph{Let $S=\POm^-(d,q)$, with $d$ even and $d\geq 4$.}
\par
Since $\POm^-(4,q)\cong \PSL(2,q^2)$ and $\POm^-(6,q)\cong \PSU(4,q)$,
we may assume that $d\geq 8$. Hence we are left, by Proposition~\ref{prop2.7}, only 
with the cases $\POm^-(8,q)$, for $q\geq2$, and $\POm^-(10,2)$.
\par
Let $S=\POm^-(8,q)$. For $S=\POm^-(8,2)$, it follows from \cite[p.\,89]{ATLAS} that $7,17\in oe(S)$
and each maximal subgroup of $S$ of
order divisible by $17$ is isomorphic to $\PSL(2,16):2$, and hence contains no elements of order $7$. 
Thus $\langle x,y\rangle=S$ whenever $x,y\in S$ with $|x|=7$ and $|y|=17$.
So we may assume that $q>2$. Then, by Theorem~\ref{Z}, both $\bbppd(q,8)$ and $\bbppd(q,6)$ are nonempty. 
Let $a\in \bbppd(q,8)$ and $b\in \bbppd(q,6)$.
Since $(q^4+1)(q^6-1)$ divides $|S|$, it follows by Remark~\ref{Rem 2.3} that $a,b\in oe(S)$. Consider now
$\hat S=\Omega ^-(8,q)\leq \GL(8,q)$. Since $\gcd(ab,|\mathbb Z(\hat S)|)
=1$, it follows that
$a,b\in oe(\hat S)$. Let $\hat x,\hat y\in \hat S$ with $|\hat x|$ a multiple
of $a$ and $|\hat y|$ a multiple of $b$, and let $X=\langle \hat x,\hat y\rangle$.   
Then $X$ is nonsolvable by Lemma~\ref{lem2.6} with $\delta=2,j=0$. Hence $\langle x,y\rangle$
is nonsolvable whenever $x,y\in S$ with $|x|=a$ and $|y|=b$.
\par
Finally, let $S=\POm^-(10,2)\leq \GL(10,2)$. Consider $11\in \bbppd(2,10)$
and $17\in \bbppd(2,8)$, and note that 
$11,17\in oe(S)$. Then by Lemma~\ref{lem2.6} with $\delta=2,j=0$, the subgroup $\langle x,y\rangle$
is nonsolvable whenever $x,y\in S$ with $|x|=11$ and $|y|=17$.
\par
{\bf (7):} \emph{Let $S=\POm^+(d,q)$, with $d$ even.}
\par Since $S$
is not nonabelian simple if $d=2$ and $d=4$, and $\POm^+(6,q)\cong \PSL(4,q)$, we may assume that $d\geq8$. 
Then by Proposition~\ref{prop2.7}, we only need to consider the cases $\POm^+(8,q)$, for $q\geq2$, and $\POm^+(10,2)$.
\par
Let $S=\POm^+(8,q)$. For $S=\POm^+(8,2)$, it follows from \cite[p.\,85]{ATLAS} that $7,15\in oe(S)$
and each maximal subgroup of $S$ of
order divisible by $35$ is isomorphic to $\PSp(6,2)$, $2^6:A_8 $ or $A_9$.
We have shown above that a subgroup of $\PSp(6,2)$
which contains elements of orders $15$ and $7$ is nonsolvable. Also, in the last paragraph of the proof of 
Proposition~\ref{prop3.4}, we saw that a subgroup of $A_9$ containing elements of orders $7$ and $5$ is $A_d$ 
for some $d\geq7$. Thus also a subgroup of $2^6:A_8 $ containing such elements has a composition factor $A_7$ 
or $A_8$.  Hence $\langle x,y\rangle$ is nonsolvable whenever $x,y\in S$ with $|x|=7$ and $|y|=15$. Thus we may assume that $q>2$. 
\par 
For $S=\POm^+(8,3)$, it follows from \cite[pp.\,140--141 and 54--55]{ATLAS} that $7,15\in oe(S)$, each maximal subgroup of $S$ of
order divisible by $7$ is isomorphic to 
$\POm^{\circ}(7,3)$, $\POm^+(8,2)$ or $2\centerdot \PSU(4,3)\centerdot 2^2$, and the last of these groups 
contains no elements of order $15$.
Let $x,y\in S$ with $|x|=7$ and $|y|=15$ and let $X=\langle x,y\rangle$.
Assume that $X<S$ and let $M$ be a maximal subgroup of $S$ containing
$X$. Then $M$ is $\POm^{\circ}(7,3)$ or $\POm^+(8,2)$.
Suppose first that $M=\POm^{\circ}(7,3)$. Observe that $5,7\in oe(M)$, 
$7\in \bbppd(3,6)$ and $5\in \bbppd(3,4)$.
By Lemma~\ref{lem2.6} with $\delta=2, j=1$, every subgroup of $\hat M=\Omega ^{\circ}(7,3)\leq \GL(7,3)$
containing elements of orders $7$ and $5$
is nonsolvable. Hence in this case $X$ is nonsolvable. On the other hand, if 
$M=\POm^+(8,2)$, then by the previous
paragraph any subgroup of $M$ containing elements of orders $7$ and $15$
is nonsolvable. Thus also in this case $X$ is nonsolvable.
So we may assume that $q\geq4$.
\par
Next we deal with $S=\POm^+(8,5)$. Notice that both $7\in \bbppd(5,6)$ and $13\in \bbppd(5,4)$ belong to $oe(S)$. 
Let $X=\langle \hat x,\hat y\rangle\leq \Omega ^+(8,5)\leq\GL(8,5)$, where $|\hat x|$ is divisible by $7$ and $|\hat y|$
is divisible by $13$. We shall prove that $X$ is nonsolvable. 
Now $\hat x$ is a $\bbppd(5,6)$-element of $X$ (even if not a large one), and  
$\hat y$ is an $\Lbpd(5,4)$-element of $X$. Assume first that $X$ acts reducibly on
$V=V(8,5)$. By Remark~\ref{Rem 2.5}(2) there exists an $X$-composition factor 
$U$ of $V$ of dimension $d_0=\dim_{\mathbb F_5}(U)\geq 6$ such that $|\hat x^U|_7\geq7$ and $|\hat y^U|_{13}\geq13$. 
Thus $d_0\in \{6,7\}$, and by Lemma~\ref{lem2.6} the group $X^U$ is nonsolvable in both cases $d_0=6$ (taking $\delta=2,j=0$)
and $d_0=7$ (taking $\delta=2,j=1$).
\par
Thus we may assume that $X$ is an irreducible subgroup of $\GL(8,5)$. Observe that here we cannot  use Lemma~\ref{lem2.6}, 
or even Lemma~\ref{lem3.1}, because $\Lbpd(5,6)=\emptyset$. So we have to 
apply \cite[Theorem 3.1]{NP2} directly (with $d=8, e=6$), checking each of the cases (a)--(e) of that theorem. If
$X$ is in case (a) or (e), then $X$ is nonsolvable. Since $7\cdot 13$ divides $|X|$, neither of the cases (b) nor (d) 
holds for $X$. So we may assume that case (c) holds for $X$, which implies that $X$ is (isomorphic to) a subgroup
of $\GL(4,5^2)\cdot 2$. Hence $X_0:=\langle \hat x^2,\hat y^2\rangle\leq \GL(4,5^2)$.
Observe that $X_0$ acts irreducibly on $V(4,5^2)$, as it acts irreducibly on $V$.
Now, $7\in \Lbpd(5^2,3)$, and hence we can apply Lemma~\ref{lem3.1}. Note that $X_0$ is not a subgroup of 
$\GL(2,5^4)\cdot 2$, because $|\GL(2,5^4)\cdot 2|$ is not divisible
by $7$. Therefore we conclude that  $X_0$ is nonsolvable, and hence also  $X$
is nonsolvable.
\par 
Now consider $S=\POm^+(8,q)$ with $q\neq 2,3,5$.  By Theorem~\ref{Z} and Proposition~\ref{Lbpd}, 
both $\Lbpd(q,6)$ and $\bbppd(q,4)$ are 
nonempty. Let $a\in \Lbpd(q,6)$ and $b\in \bbppd(q,4)$.
Since $(q^6-1)(q^4-1)$ divides $|S|$, it follows by Remark~\ref{Rem 2.3} that $a,b\in oe(S)$. Consider now
$\hat S=\Omega ^+(8,q)\leq \GL(8,q)$ acting on  on $V=V(8,q)$. 
Since $\gcd(ab,|\mathbb Z(\hat S)|)=1$, it follows that
$a,b\in oe(\hat S)$. Let $\hat x,\hat y\in \hat S$ with $|\hat x|$ a multiple
of $a$ and $|\hat y|$ a multiple of $b$, and let $X=\langle \hat x,\hat y\rangle$. 
Now $\hat x$ is an $\Lbpd(q,6)$-element and $\hat y$ is
a $\bbppd(q,4)$-element of $X$. By Remark~\ref{Rem 2.5}(2), there is  an 
$X$-composition factor $U$ of $V$ with $d_0=\dim _{\mathbb F_q}(U)\in \{6,7,8\}$ such that $|\hat x^U|_a=|\hat x|_a\geq a$ 
and $|\hat y^U|_b=|\hat y|_b\geq b$. It follows by Lemma~\ref{lem2.6} that $X^U$, and hence also $X$, is nonsolvable 
if $d_0=6$ (taking $\delta=2,j=0$), and also if $d_0=7$ (taking $\delta=2, j=1$).
Thus we may assume that $d_0=8$, that is, $X$ is irreducible on $V$. Then by 
Lemma~\ref{lem3.1}, either $X$ is nonsolvable, as required, or $X$ is (isomorphic to) a subgroup of $\GL(4,q^2).2$. Suppose that
$X\leq \GL(4,q^2).2$ and consider $X_0=\langle \hat x^2,\hat y^2\rangle\leq \GL(4,q^2)$. By the above argument applied to $X_0$, 
we may assume that $X_0$ acts irreducibly on $V=V(4,q^2)$. It follows
then, by Lemma~\ref{lem3.1}, that either $X_0$ is nonsolvable or $X_0$ is (isomorphic to) a subgroup of $\GL(2,q^4).2$. 
However, since $a\in \Lbpd(q,6)$, $a$ is coprime to $|\GL(2,q^4).2|=2q^4(q^4-1)(q^8-1)$, and we conclude that $X_0$, 
and hence also $X$, is nonsolvable. 
\par
Finally let $S=\POm^+(10,2)$. By \cite[p.\,147]{ATLAS}, $17,31\in oe(S)$ and no maximal
subgroup of $S$ has order divisible by both $17$ and $31$. Hence 
$\langle x,y\rangle=S$ whenever $x,y\in S$ with $|x|=17$ and $|y|=31$.
\end{proof}

We now prove Theorem~B for the exceptional finite simple groups of Lie type. We make use
of five papers. The first is the paper \cite{FS} of Feit and Seitz. They prove, in \cite[Theorem 3.1]{FS}, 
the existence of certain self-centralizing cyclic maximal tori in simple groups of Lie type.
The second is the paper \cite{W} of Wiegel. He gives, in \cite[Table 1]{W}, a list of cyclic maximal tori
in exceptional groups of Lie type, with some small cases excluded, 
and in \cite[Section 4]{W} he determines the maximal subgroups
containing these tori, for each such group. We were kindly informed by Frank L\"ubeck, in a letter,
that Weigel's list of cyclic maximal tori is correct without the extra conditions on $q$ or $k$, with only 
one exception: namely the group $G_2(2)$, which has elements of order $q^2-q+1=3$ in several classes
of maximal tori. We shall refer to \cite{L} concerning this important information.
The third is the paper \cite{GM} of Guralnick and Malle, which is still in preparation.
The authors kindly informed us that their paper contains important information about
maximal subgroups of $E_7(2)_{sc}$ and $E_7(3)_{sc}$.
The forth is the paper \cite{MT} of Moret\`o
and Tiep. We use \cite[Lemma 2.3]{MT}, in a slightly `extended' form, which was kindly approved by Pham Tiep.
They prove, in Lemma 2.3, that each exceptional simple group of Lie type contains elements
$s_1$ and $s_2$ of prime orders $p_1$ and $p_2$, respectively, such that their centralizers have suitable 
orders. The `extended' version of this lemma states not only that such elements exist, but also
that the centralizers of every element of order $p_1$ or $p_2$ are of the same suitable orders.
Finally, the fifth paper is the paper \cite{GK} by Guralnick and Kantor, which provides in \cite[Proposition 6.2]{GK}  
information concerning elements of the groups excluded in \cite{W} and of the sporadic subgroups, contained in a unique 
or  small number of maximal subgroups.      
\par
We recall that every finite simple group of Lie type occurs as a composition factor of
the group of fixed points $G_F$, under a Frobenius map $F:G\to G$ of a connected reductive algebraic group $G$ 
over the algebraic closure $\overline {\mathbb F_q}$ 
of a field $\mathbb F_q$ of order $q$. 
\par
If we choose $G$ to be simply connected, then every finite simple exceptional group
of Lie type is a quotient $G_F/\mathbb Z(G_F)$. Moreover, $\mathbb Z(G_F)=1$ unless $G$ is of type $E_6$,
$^2E_6$ or $E_7$.
The following facts are used repeatedly.

\begin{lemma}\label{lemtorus}
Let $S=G_F/\mathbb Z(G_F)$, $q$ be as above, and suppose that $S$ has a cyclic maximal torus $T$ of order divisible 
by a prime $p$, such that $|S:T|$ is coprime to $p$, and  $C_S(y)=T$ for $y\in T$ of order $|T|_p$. 
Then for each $x\in S$ with $|x|=|T|$, the subgroup $\langle x\rangle$ is conjugate to $T$ in $S$, 
and in particular it is a maximal torus of order $|T|$.  
\end{lemma}

\begin{proof}
By assumption $T$ has a unique  Sylow $p$-subgroup, say $P=\langle y\rangle$, and $P$ is a Sylow $p$-subgroup of $S$.  
Let $x\in S$ with $|x|=|T|$. Then $\langle x\rangle$ contains a subgroup $P_0$ of order $|P|$, 
so by Sylow's Theorem $P_0^g=P$ for some $g\in S$. Then $\langle x\rangle^g \leq C_S(y)$ which by assumption 
is equal to $T$. It follows that $\langle x\rangle^g=T$. 
\end{proof}

The following is an immediate corollary of Lemma~\ref{lemtorus}.

\begin{corollary}\label{lemendgame}
Let $S, T$ be as in Lemma~{\rm\ref{lemtorus}}, let $a=|T|\in oe(S)$, and suppose that
$b\in oe(S)$ is such that each maximal subgroup of $S$ containing $T$ has order coprime to $b$. Then for each $x,y\in S$ with $|x|=a$ and $|y|=b$, the group $\langle x,y\rangle = S$ and, in particular, is non-solvable.
\end{corollary}

We now prove
\begin{proposition}\label{prop3.3}  
Theorem~B holds for all exceptional finite simple groups of Lie type.
\end{proposition}
\begin{proof}
In the following, we denote by $\pi (n)$ the set of prime divisors of the positive
integer $n$ and by $\Phi _k(x)$ the $k$-th cyclotomic polynomial.
We consider the exceptional groups, beginning with those of smallest Lie rank.
Our basic proof strategy is to choose $a,b\in oe(S)$, where possible, so that the 
hypotheses of Corollary~\ref{lemendgame} hold. Then we have immediately that 
Theorem~B holds for $a,b$. We call this `the standard argument'.
\par 
{\bf (1):} \emph{Let $S={}^2B_2(q)$, with $q=2^{2n+1}$ and $n\geq 1$.}
\par Then $|S|=q^2(q-1)(q^2+1)$. 
Write $r=2^{n+1}$, so $q^2+1=(q+r+1)(q-r+1)$. 
Since $^2B_2(2)$ is a Frobenius group of order $20$, the field order $q$ is at least 8.
\par
Let $p\in \ppd(q,4)$, which is nonempty. Since $q^2+1=(q+r+1)(q-r+1)$, the prime $p$
divides $a:=q+\varepsilon r+1$ where $\varepsilon=\pm1$. By \cite[Theorem 3.1]{FS}, $S$ has a cyclic maximal torus $T$  
of order $a$ and we note that $|S:T|$ is coprime to $p$.
By comparing orders, we deduce from 
the `extended' \cite[Lemma 2.3]{MT} that $C_S(y)=T$ for $y\in T$ of order $|T|_p$.
By \cite{S}, the only maximal 
subgroup of $S$ containing $T$ is its normaliser, of order $4a$. 
Then the standard argument applies for $a$ and any $b\in \pi(q-\varepsilon r+1)$, 
since $b\neq 2$ and $\gcd(q+r+1,q-r+1)=1$, so $b$ does not divide $4a$.
\par
{\bf (2):} \emph{Let $S={}^2G_2(q)'$, with $q=3^{2n+1}$ and $n\geq 1$.}
\par Then $|S|=q^3(q-1)(q^3+1)$.
Write $r=3^{n+1}$, so $q^3+1=(q+1)(q+r+1)(q-r+1)$.
Since $^2G_2(3)'\cong \PSL(2,8)$ has been already
treated in Proposition~\ref{prop3.2}, we may assume that $q\geq 27$.
\par
Let $p\in \ppd(q,6)$, which is nonempty. Then $p$ divides $a=q+\varepsilon r+1$ where $\varepsilon=\pm1$. 
By \cite[Theorem 3.1]{FS}, 
$S$ has a cyclic maximal torus $T$
of order $a$ and we note that $|S:T|$ is coprime to $p$. As in (1), 
 $C_S(y)=T$ for $y\in T$ of order $|T|_p$, and
by \cite{K} and \cite{LN}, the only maximal 
subgroup of $S$ containing $T$ is its normaliser, of order $6a$.
Let $b\in \pi(q-\varepsilon r+1)$.
Then $b\in oe(S)$, but $b$ does not divide $6a$, because $b\ne 2,3$ and $\gcd(q+r+1,q-r+1)=1$. Thus the standard argument applies.
\par

{\bf (3a):} \emph{Let $S={}^2F_4(2)'$, the Tits group.}
\par Then $|S|=2^{11}\cdot 3^3\cdot 5^2\cdot 13$.
By \cite{ATLAS3}, $13,10\in oe(S)$ and each maximal
subgroup of $S$ of order divisible by $130$ is isomorphic to $\PSL(2,25)$, which
contains no elements of order $10$. Hence $\langle x,y\rangle=S$ whenever $x,y\in S$
with $|x|=13$ and $|y|=10$.
\par
{\bf (3b):} \emph{Let $S={}^2F_4(q)$, with $q=2^{2n+1}$ and $n\geq 1$.}
\par Then 
$|S|=q^{12}(q^6+1)(q^4-1)(q^3+1)(q-1)$. Write $r=2^{n+1}$, so 
$$\frac {q^6+1}{q^2+1}= q^4-q^2+1=(q^2+rq+q+r+1)(q^2-rq+q-r+1)
$$
and $\gcd(q^2+rq+q+r+1,q^2-rq+q-r+1)$
divides $(q^2+q+1)(q-1)=q^3-1$.
\par
Let $p\in \ppd(q,12)$, which is nonempty. Then $p$ divides $a:=q^2+\varepsilon rq+q+\varepsilon r+1$, where $\varepsilon=\pm1$, 
and $|S|/a$ is coprime to $p$. By \cite[Theorem 3.1]{FS}, $S$ has a cyclic maximal torus $T$ of order $a$, and arguing as in (1), the hypotheses of Lemma~\ref{lemtorus} hold for $T$. By \cite{M}, the only maximal
subgroup of $S$ containing $T$ is $N_S(T)$ of order $12a$. 
Let $b\in\ppd(q,6)$, which is nonempty. Then $b$ divides $q^3+1$ and hence $b\in oe(S)$ and $b$ does not divide $a$. 
Also $b\geq 7$ and so $b$ does not divide $12$. 
It follows that $b$ does not divide $12a$,
and hence, the standard argument applies.
\par
{\bf (4):} \emph{Let $S=G_2(q)$, with $q>2$.}
\par Then $|S|=q^6(q^6-1)(q^2-1)$.
Since $G_2(2)'\cong \PSU(3,3)$ has been already treated, we may assume that $q\geq 3$.
First we deal with $G_2(3)$ and $G_2(4)$, which were excluded in \cite{W}.
\par
Let $S=G_2(q)$ with $q=3$ or $4$. Then by \cite[pp.\,60--61,97]{ATLAS}, $a,13\in oe(S)$,
where $a=7$ if $q=3$ and $a=5$ if $q=4$, and each
maximal subgroup $M$ of $S$ of order divisible by $13a$ is isomorphic to 
$\PSL(2,13)$ if $q=3$, and to $\PSU(3,4):2$ if $q=4$. In either case, the derived group $M'$ is generated 
by any pair of its elements with one of order $a$ and the other of order $13$. Hence, if $x,y\in S$ 
with $|x|=a$ and $|y|=13$, then
$\langle x,y\rangle$ is nonsolvable.
\par
Let, now,  $S=G_2(q)$, with $q\geq 5$ and let $p\in \ppd(q,6)$, which is nonempty. Since $q^3+1=(q^2-q+1)(q+1)$,
$p$ divides $a:=q^2-q+1=\Phi_6(q)$ and $|S|/a$ is coprime to $p$.  By \cite[Table I]{W}, $S$ has a cyclic maximal torus $T$ 
of order $a$, and arguing as in (1), the hypotheses of Lemma~\ref{lemtorus} hold for $T$. By \cite[Section 4]{W},
each maximal subgroup $M$ of $S$ containing $T$ is isomorphic to $\SU(3,q).2$ and hence
$|M|=2q^3(q^3+1)(q^2-1)$. Let $b\in \ppd(q,3)$, which is nonempty. Then $b\in oe(S)$ and $\gcd(b,2q)=1$.
Since $b$ divides $q^3-1$, also $\gcd(b,q^3+1)=1$, and it follows that $b$ does not divide $|M|$. Hence, the standard argument applies.


\par
{\bf (5):} \emph{Let $S={}^3D_4(q)$, with $q\geq 2$.} 
\par Then $|S|=q^{12}(q^8+q^4+1)(q^6-1)(q^2-1)$, where
$q^8+q^4+1=(q^4-q^2+1)(q^4+q^2+1)$.
\par
Let $p\in \ppd(q,12)$, which is nonempty. Since $q^6+1=(q^4-q^2+1)(q^2+1)$, 
$p$ divides $a=q^4-q^2+1=\Phi _{12}(q)$ and $|S|/a$ is coprime to $p$.
By \cite[Table I]{W}, there exist a cyclic maximal torus $T$ of $S$
of order $a$, and arguing as in (1), the hypotheses of Lemma~\ref{lemtorus} hold for $T$. By \cite[Section 4]{W},
the only maximal subgroup of $S$ containing $T$ is $N_S(T)$ of order $4a$.
Let $b\in \ppd(q,6)$ if $q\neq 2$ and let $b=7$ if $q=2$. Then $b\neq 2$, $b\in oe(S)$ and $\gcd(b,q^4-q^2+1)=1$,
since $b$ divides $q^6-1$ and $q^4-q^2+1$ divides $q^6+1$. Thus $b$ does not divide $4a$, and hence, the standard argument applies.
\par
{\bf (6):} \emph{Let $S=F_4(q)$, with $q\geq 2$.}
\par  Then $|S|=q^{24}(q^{12}-1)(q^8-1)(q^6-1)(q^2-1)$.
First we deal with $F_4(2)$ and $F_4(3)$, which were excluded in \cite{W}.
\par
Let first $S=F_4(2)$. Notice that $13,17\in oe(S)$ and by \cite[Proposition 6.2]{GK},
the maximal subgroups of $S$ of order divisible by $17$ are isomorphic to $\PSp(8,2)$, which is of order not
divisible by $13$. Hence, the standard argument applies.
\par
Let  now $S=F_4(3)$. Notice that both  $73\in \ppd(3,12)$ and $41\in \ppd(3,8)$ belong to $oe(S)$.
By \cite[Proposition 6.2]{GK},
the maximal subgroups of $S$ of order divisible by $73$ are isomorphic to $^3D_4(3).3$, which is of order not
divisible by $41$ (see \cite[p. 241]{ATLAS}). Hence, the standard argument applies.
\par 
Let, finally, $S=F_4(q)$, with $q\geq 4$ and let $p\in \ppd(q,12)$. It follows, as in (5), that 
$p$ divides $a=q^4-q^2+1=\Phi _{12}(q)$ and $|S|/a$ is coprime to $p$.
By \cite[Table I]{W}, $S$ has a cyclic maximal torus $T$ 
of order $a$ and arguing as in (1), the hypotheses of Lemma~\ref{lemtorus} hold for $T$. By \cite[Section 4]{W},
every maximal subgroup $M$ of $S$ containing $T$ is isomorphic to $^3D_4(q).3$ and hence
$|M|=3q^{12}(q^8+q^4+1)(q^6-1)(q^2-1)$.
Let $b\in \ppd(q,8)$, which is nonempty. Then $b\in oe(S)$, but $b$ does not divide $|M|$, 
since $\gcd(b,3q)=1$, $\gcd(b,(q^6-1)(q^2-1))=1$  and $\gcd(q^8+q^4+1,b)$ divides $\gcd(q^{12}-1,q^8-1)=q^4-1$,
so also $\gcd(q^8+q^4+1,b)=1$. Hence, the standard argument applies.
\par
{\bf (7):} \emph{Let $S={}^2E_6(q)$, with $q\geq 2$.}
\par Then $|S|=\frac 1dq^{36}(q^{12}-1)(q^9+1)(q^8-1)(q^6-1)(q^5+1)(q^2-1)$,
and $d=\gcd(3,q+1)$. Moreover, $S=\hat S/\mathbb Z(\hat S)$, where $\hat S=G_F$, with $G$ a simply connected
algebraic group of exceptional type $^2E_6$ and $|\mathbb Z(\hat S)|=d$.
\par
Let $p\in \ppd(q,18)$, which is nonempty. Since $q^9+1=(q^3+1)(q^6-q^3+1)$, $p$ divides $a=q^6-q^3+1$.
By \cite[Table 1]{W} and \cite{L}, $\hat S$ has a cyclic maximal torus $T$ 
of order $a$ and arguing as in (1), the hypotheses of Lemma~\ref{lemtorus} hold for $T$. Note that $\mathbb Z(\hat S)\leq T$.
\par
It follows by \cite[Section 4]{W} 
for $q\geq 4$ and by \cite[Theorem 6.2]{GK} for $q=2,3$,
that every maximal subgroup $M$ of $\hat S$ containing $T$ is isomorphic to $\PSU(3,q^3).3$ and hence
$|M|=\frac 3dq^9(q^9+1)(q^6-1)$.
\par
Let $b\in \ppd(q,12)$, which is nonempty, and note that 
$b\geq 12+1=13$. Then $b$ divides $q^6+1$ and $b\neq 3 $.
Hence $b\in oe(\hat S)$ and $b$ does not divide $|M|$. Thus, by the standard argument,
if $x\in \hat S$
is of order $a$ and $y\in \hat S$ is of order $b$,
then $\langle x,y\rangle=\hat S$.
\par 
If $d=1$, then Theorem~B holds for $S=\hat S$. So suppose that $d=3$. Then $q\equiv -1 \pmod{3}$, 
$a=q^6-q^3+1\equiv 3 \pmod{9}$, $\gcd(3,b)=1$ and
since $\mathbb Z(\hat S)\leq \langle x\rangle$ for each $x\in \hat S$ of order
$a$, it follows that $3$ divides $a$. Thus $a/3,b\in oe(S)$.
Let $z=\hat z\mathbb Z(\hat S)$ be an arbitrary element of $S$ of order $a/3$ and let
$w=\hat w\mathbb Z(\hat S)$ be an arbitrary element of $S$ of order $b$,
where $\hat z,\hat w$ are elements of
$\hat S$. Since $\gcd(a/3,3)=\gcd(b,3)=1$ and $|\mathbb Z(\hat S)|=3$, we may always
choose $\hat z$ of order $a$  and $\hat w$ of order $b$. Since, as shown above,
$\langle \hat z,\hat w\rangle=\hat S$, it follows that
$\langle z,w\rangle=S$.
\par       
{\bf (8):} \emph{Let $S=E_6(q)$, with $q\geq 2$.}
\par Then $|S|=\frac 1dq^{36}(q^{12}-1)(q^9-1)(q^8-1)(q^6-1)(q^5-1)(q^2-1)$,
where $d=\gcd(3,q-1)$. Moreover, $S=\hat S/\mathbb Z(\hat S)$, where $\hat S=G_F$, with $G$ a simply connected 
algebraic group of exceptional type $E_6$ and $|\mathbb Z(\hat S)|=d$.
\par
Let $p\in \ppd(q,9)$, which is nonempty. Since $q^9-1=(q^3-1)(q^6+q^3+1)$, $p$ divides $a=q^6+q^3+1$.
By \cite[Table 1]{W}, $\hat S$ has a cyclic maximal torus $T$ 
of order $a$, and arguing as in (1), the hypotheses of Lemma~\ref{lemtorus} hold for $T$. 
In particular, $\mathbb Z(\hat S)\leq T$. 
By \cite[Section 4]{W},
every maximal subgroup $M$ of $\hat S$ containing $T$ is isomorphic to $\SL(3,q^3).3$ and hence
$|M|=3q^9(q^9-1)(q^6-1)$.
\par
Let $b\in \ppd(q,12)$, which is nonempty, and note that $b\geq 12+1=13$. Then $b\in oe(\hat S)$, $b$ divides $q^6+1$ and $b\neq 3 $.
It follows that $b$ does not divide $|M|$.
Hence, by the standard argument, if $x\in \hat S$ 
is of order $a$ and $y\in \hat S$ is of order $b$,
then $\langle \hat x,\hat y\rangle=\hat S$.
\par
If $d=1$, then Theorem~B holds for $S=\hat S$. So suppose that $d=3$. Then it follows using the same proof as for ${}^2E_6(q)$ that if 
$z=\hat z\mathbb Z(\hat S)$ has order $a/3$ and $w=\hat w\mathbb Z(\hat S)$ has order $b$, where $\hat z,\hat w$ are elements of $\hat S$, then  $\langle z,w\rangle=S$.
\par
{\bf (9):} \emph{Let $S=E_7(q)$, with $q\geq 2$.}
\quad Then 
$$
|S|=\frac 1dq^{63}\prod _{i\in I}(q^i-1),\qquad \text{with}\quad I=
\{2,6,8,10,12,14,18\}\,
$$
where $d=\gcd(2,q-1)$. Moreover, $S=\hat S/\mathbb Z(\hat S)$, where $\hat S=G_F$, with $G$ a simply connected
algebraic group of exceptional type $E_7$ and $|\mathbb Z(\hat S)|=d$.
\par
Let $p\in \ppd(q,18)$, which is nonempty. Since $q^9+1=(q^3+1)(q^6-q^3+1)$, $p$ divides $a=(q+1)(q^6-q^3+1)$.
By \cite[Table 1]{W} and \cite{L}, $S$ has a cyclic maximal torus $T$ 
of order $a$, and arguing as in (1), the hypotheses of Lemma~\ref{lemtorus} hold for $T$. In particular, $\mathbb Z(\hat S)\leq T$.
\par
By \cite[Section 4]{W} for $q\geq 4$ and by \cite[Proposition 2.11]{GM} for $q=2,3$,
every maximal subgroup $M$ of $\hat S$ containing $T$ is isomorphic to $(Z_{q+1}\cdot^2E_6(q)).2$ and hence
$|M|=\frac2{d_1}(q+1)q^{36}(q^{12}-1)(q^9+1)(q^8-1)(q^6-1)(q^5+1)(q^2-1)$, where $d_1=\gcd(3,q-1)$. 
\par
Let $b\in \ppd(3,14)$, which is nonempty. Then $b\in oe(\hat S)$, $b$ divides $q^7+1$ and $b\neq 2$.
Since $\gcd(q^9+1,q^7+1)$ divides $q^2-1$, it follows that $b$ does not divide $|M|$.
Hence, by the standard argument, if $x\in \hat S$
is of order $a$ and $y\in \hat S$ is of order $b$,
then $\langle \hat x,\hat y\rangle=\hat S$.
\par
If $d=1$, then Theorem~B holds for $S=\hat S$. So suppose that $d=2$. Then $q$ is odd and 
since $\mathbb Z(\hat S)\leq \langle x\rangle$ for each $x\in \hat S$ of order
$a$, it follows that $a$ is even, and $a/2,b\in oe(S)$.
Let $z=\hat z\mathbb Z(\hat S)$ be an arbitrary element of $S$ of order $a/2$ and let
$w=\hat w\mathbb Z(\hat S)$ be an arbitrary element of $S$ of order $b$,
where $\hat z,\hat w$ are elements of $\hat S$. 
Since $\gcd(b,2)=1$ and $|\mathbb Z(\hat S)|=2$, it follows that we may 
choose $\hat w$ of order $b$. Now consider $\hat z$. Let  $L=\langle \hat z,\mathbb Z(\hat S)\rangle$.
Then $L$ is abelian and it contains an element $\hat u$ of order $p$. By \cite[Lemma 2.3]{MT}, 
$|\mathbb C_{\hat S}(\hat u)|=a$ and as shown in (1), this centralizer is a cyclic maximal torus 
of $\hat S$. Consequently, as $L\leq \mathbb C_{\hat S}(\hat u)$, we may also choose $\hat z$ of order $a$.  
Since, as shown above,
$\langle \hat z,\hat w\rangle=\hat S$, it follows that
$\langle z,w\rangle=S$.
\par
{\bf (10):} Let \emph{$S=E_8(q)$, with $q\geq 2$.}
\quad Then 
$$|S|=q^{120}\prod _{i\in I}(q^i-1),\qquad \text{with}\quad I=
\{2,8,12,14,18,20,24,30\}\ .$$
Let $p\in \ppd(q,30)$, which is nonempty. Since 
\begin{multline}
$$q^{15}+1=(q^2-q+1)(q^5+1)(q^8+q^7-q^5-q^4-q^3+q+1)\\
=(q^2-q+1)(q^5+1)\Phi_{30}(q),$$
\end{multline}
$p$ divides $a=\Phi _{30}(q)$ and $|S|/a$ is coprime to $p$.
By \cite[Table I]{W}, $S$ has a cyclic maximal torus $T$ 
of order $a$, and arguing as in (1), the hypotheses of Lemma~\ref{lemtorus} hold for $T$. By \cite[Section 4]{W},
the only maximal subgroup of $S$ containing $T$ is $N_S(T)$ of order $30a$.
\par
Let $b\in \ppd(q,24)$, which is nonempty, and note that $b\equiv 1\pmod{24}$, and hence that $b\geq 25$ and $b$ does not divide $30$. 
Now $b$ divides $q^{12}+1$ and $a$ divides $q^{15}+1$, and
since $\gcd(q^{12}+1,q^{15}+1)$ divides $q^3-1$,
it follows that $b$ does not divide $a$. Therefore, $b$ does not divide $|N_S(T)|=30a$, so the standard argument applies.
\end{proof}


We are now ready to complete the proof of Theorem~B, which we state again:
\begin{theoremB}
Let $S$ be a nonabelian finite simple group. Then there exist
$a,b \in oe(S)$,  such that every pair of elements of $S$ of order
$a$ and $b$, respectively, generates a nonsolvable subgroup of $S$.
\end{theoremB}
\begin{proof}
This follows from
Propositions~\ref{prop3.4}, \ref{prop3.5}, \ref{prop3.2}, \ref{prop3.3} and from the classification of the finite
simple groups. 
\end{proof}
\par
\section{Proof of Theorem~A}

We finally show that Theorem~A follows from Theorem~B.
First, we restate Theorem~A.
\begin{theoremA}
Let $G$ be a finite group. Assume that for every $x,y \in G$ there
exists an element $g \in G$ such that $\langle x,y^g\rangle$ is solvable.
Then $G$ is solvable.
\end{theoremA}
\begin{proof}
Suppose that the hypothesis holds for a group $G$. This is clearly equivalent to assuming that,
for all pairs $C,D$ of conjugacy classes of a group $G$ there
exist elements $x \in C$ and $y\in D$ such that $\langle x,y\rangle$ is solvable.
\par
We claim that this property is inherited by factor groups:
let $N$ be a normal subgroup of $G$ and write $\overline{G} = G/N$.
Since ``overbar'' is a homomorphism, it sends conjugacy classes of $G$
onto conjugacy classes of $\overline{G}$.
Hence, given two  conjugacy classes $\overline{C}$ and $\overline{D}$ of
$\overline{G}$, we may assume that $C$ and $D$ are conjugacy classes of $G$.
So, by our assumption, there exist $x \in C$ and $y\in D$ such that $\langle x,y\rangle$ is solvable.
Hence, $\overline{x} \in \overline{C}$, $\overline{y} \in \overline{D}$ and
$\langle \overline{x} , \overline{y}\rangle \leq \overline{\langle x,y\rangle}$ is solvable. 
Thus the claim is proved.
\par
Theorem~A holds trivially if $|G|=1$. Suppose inductively 
that $|G|>1$ and that Theorem~A holds for groups of orders 
less than $|G|$. Let $M$ be a minimal normal subgroup of $G$.
Then by induction, $G/M$ is solvable. 
If $G$ has distinct minimal normal
subgroups $M_1, M_2$, then $G$ is isomorphic to a subgroup of 
the solvable group $G/M_1 \times G/M_2$.
Thus we may assume that $G$ has a unique
minimal normal subgroup $M$. 
If $M$  is solvable, then $G$ is solvable as well. 
We show that this must be the case: suppose to the contrary that $M$ is nonsolvable.
The characteristically simple group $M$ is a direct product of isomorphic simple groups.
We hence identify $M$ with the direct power $S^k$ of a nonabelian simple group $S$.
By Theorem~B, there exist $a,b \in oe(S)$ such that for every choice  of elements
$x, y \in S$ with $|x| = a$ and $|y| = b$, the group $\langle x,y\rangle$ is nonsolvable.
In particular,
$\langle x^\alpha, y^\beta\rangle$ is nonsolvable for all $\alpha,\beta \in \text{Aut}(S)$.
Consider now the diagonal elements $u =(x,x, \dots, x)$ and $w = (y, y, \dots, y)$ of
$M$. Recalling that $G$ can be embedded in the wreath product of $\text{Aut}(S)$ by a
(solvable) subgroup of the symmetric group $S_k$, we see that for every $g, h \in G$ we have
$u^g =(x^{\alpha_1},x^{\alpha_2}, \dots, x^{\alpha_k})$ and
$w^h = (y^{\beta_1}, y^{\beta_2}, \dots, y^{\beta_k})$,
where $\alpha_i, \beta_i \in \text{Aut}(S)$ for $i = 1, 2, \dots, k$.
Observe also that $\langle u^g, w^h\rangle$ is a subdirect subgroup of the direct product
$$
 \prod_{i=1}^k \langle x^{\alpha_i}, y^{\beta_i}\rangle.
$$
However $\langle x^{\alpha_i}, y^{\beta_i}\rangle$ is a nonsolvable subgroup of $S$, for every $i = 1, \dots, k$.
It follows that $\langle u^g, w^h\rangle$ is nonsolvable
for every choice of $g$ and $h$ in $G$, which is the required contradiction.
\end{proof}


\end{document}